\documentclass[a4paper, intlimits, 10pt,reqno]{amsart}

\usepackage[T1]{fontenc}
\usepackage{float}

\usepackage{times}
\usepackage[english]{babel}
\usepackage[T1]{fontenc}
\usepackage{enumitem}

\usepackage{graphicx}

\usepackage{subcaption}
\usepackage{tabularx}
\usepackage{color}
\usepackage{eurosym}
\usepackage[numbers, sort&compress]{natbib}

\usepackage{amsthm, amssymb, amsfonts}
\usepackage{amsmath}

\usepackage{accents}

\usepackage{multirow}

\newtheorem{theorem}{Theorem}[section]
\newtheorem{corollary}[theorem]{Corollary}

\newtheorem{lemma}[theorem]{Lemma}

\newtheorem{definition}[theorem]{Definition}

\newtheorem{remark}[theorem]{Remark}

\numberwithin{equation}{section}

\begin{document}

\title{$q$-Moment Estimates for the Singular $p$-Laplace Equation and Applications}

\date{\today}

\thanks{National Science Foundation of China, Grants Numbers: 11971310 and 11671257 are gratefully acknowledged.}

\author{Samuel Drapeau}
\address{Shanghai Jiao Tong University, China Academy of Financial Research (SAIF) and School of Mathematical Sciences, Shanghai, China}
\email{sdrapeau@saif.sjtu.edu.cn}
\urladdr{http://www.samuel-drapeau.info}

\author{Liming Yin}
\address{Shanghai Jiao Tong University, School of Mathematical Sciences, Shanghai, China}
\email{gacktkaga@sjtu.edu.cn}

\begin{abstract}
    We provide $q$-moment estimates on annuli for weak solutions of the singular $p$-Laplace equation where $p$ and $q$ are conjugates.
    We derive $q$-uniform integrability for some critical parameter range.
    As a application, we derive a mass conservation as well as a weak convergence result for a larger critical parameter range.
    Concerning the latter point, we further provide a rate of convergence of order $t^{q-1}$ of the solution in the $q$-Wasserstein distance.
\end{abstract}

\maketitle
\section{Introduction}
We consider the Cauchy problem for the parabolic $p$-Laplace equation posed in the whole Euclidean space $\mathbb{R}^N$
\begin{equation}
    \label{eq:p_laplace}
    \begin{cases}
        \displaystyle \partial_t u = \mathrm{div}(|\nabla u|^{p-2} \nabla u),\\
        u(0)=\mu,
    \end{cases}
\end{equation}
with a positive Radon measure $\mu$ as initial data and $p_c<p<2$ where the critical value $p_c$ is defined as 
\begin{equation*}
    p_c := \frac{2N}{N+1}.
\end{equation*}
Solutions are understood in the weak sense on the whole space and finite time horizon $S_T:=\mathbb{R}^N \times (0,T]$ where $N\geq 2$.

For $1<p<2$, such Cauchy problem is know as the fast diffusion or singular equation.
Existence and uniqueness and regularity of weak solutions with various conditions for initial data are studied by \citet{dibenedetto1990, zhao1995, chen1988}.
For further insights about parabolic $p$-Laplace equation, we refer to \citep{Benedetto1993, Vazquez2006}.
While for a Dirac measure $\delta_0$ as initial data, this Cauchy problem does not admit a positive solution for any $1<p\leq p_c$, see \citep{zhao1995}, it is known that for $p_c<p<2$, there exists a fundamental solution called Barenblatt source-type solution, see \citep{Vazquez2006}.

In the context of optimal matching problem, \citet{ambrosio2019pde} draw a link between $2$-Wasserstein distance and Poisson equation through linearization of Monge-Amp\`ere equation and regularize empirical measures with heat semigroup.
They suggest that such a linearization for the $q$-Wasserstein distance leads to $p$-Laplacian where $p$ and $q$ are conjugates.
Such an approach has been used in a different context by \citet{evans1999differential}.

Based on these ideas, it seems natural to regularize empirical measures with $p$-Laplacian semigroup.
And further, to connect the Cauchy problem \eqref{eq:p_laplace} to flows in the $q$-Wasserstein space, as done in \citet{kell2016q}.
A key point in this connection concerns the behavior of $q$-moment of weak solutions $u(x,t)$ of Cauchy problem \eqref{eq:p_laplace}, namely,
\begin{equation*}
    \int_{\mathbb{R}^N} |x|^q u(x,t)dx.
\end{equation*}
For the degenerate case where $p>2$, that is, $1<q<2$, such $q$-moments are easily estimated, at least for initial data with finite mass and compact support, due to the finite propagation property, see \citep{kamin1988}.
However, for the singular case $p_c<p<2$, solutions diffuse everywhere even with a Dirac measure as initial data.
Hence, the focus on the parameter range $p_c<p<2$.

Following similar methods as in \citep{dibenedetto1990, zhao1995}, our main result are the following $q$-moment estimate for weak solutions of \eqref{eq:p_laplace} on annuli
\begin{multline}\label{eq:introd_1}
    \sup_{0< \tau \leq t} \int_{r\leq|x|\leq R} \left| x \right|^q u(x,\tau)dx
    \leq
    C \int_{\frac{1}{2}r\leq|x|\leq 2R} |x|^q d\mu\\
    +
    C t^{\frac{1}{2-p}}\left(
        r^{-\frac{p(N)}{(p-1)(2-p)}} + R^{-\frac{p(N)}{(p-1)(2-p)}}
    \right).
\end{multline}
for any $p_c<p<2$, where $C$ is a constant depending only on $N$ and $p$ and $p(N)$ is the polynomial 
\begin{equation*}
    p(N) \colon = (N+2)p^2 - (3N+3)p + 2N. 
\end{equation*}
This estimate shows in particular that with a finite $q$-moment Radon measure as initial data, there exists a weak solution which has uniformly bounded $q$-moments on all compact time interval for any $p_N<p<2$, where the critical value $p_c<p_N<2$ is the largest root of the polynomial $p(N)$.
This critical range $(p_N,2)$ is sharp in the sense that the Barenblatt source-type solution has finite $q$-moment if and only if $p_N<p<2$.

As applications, we show that for such a weak solution with a finite $q$-moment Radon measure as initial data, the mass conservation as well as weak convergence at $t=0$ holds for any $p_c<p<2$.
Finally, we obtain a $q$-Wasserstein convergence rate for weak solutions $u$ with a finite $q$-moment Radon measure $\mu$ as initial data for any $p_N<p<2$, namely, 
\begin{equation}
    W^q_q(\mu,u(x,t)dx) \leq C t^{q-1},
\end{equation}
where $C$ is a constant depending only on $N$, $p$, $\mu$ and $T$.
In particular, when $p$ converges to $2$, the convergence rate coincides with the well-known heat flow case in \citep{ambrosio2019pde}.

To the best of our knowledge, there does not exists $q$-moment estimates for singular $p$-Laplacian equation.
However, similar results in terms of mass conservation and connection to $q$-Wasserstein space has been done.
The mass conservation of singular parabolic $p$-Laplace equation for $p_c<p<2$ is proved by \citet{fino2014} for any weak solution where the initial data are Radon measures with compact support.
Our result only requires finite $q$-moment instead of compact support for initial data, but for any weak solution constructed through mollification.
As for the connection to the flow in the $q$-Wasserstein space, the work of \citet{kell2016q} is the closest to the present one.
There, the author shows that a smooth solution of Cauchy problem \eqref{eq:p_laplace} solves a generalized gradient flow problem of R\'enyi entropy functional in the $q$-Wasserstein space based on gradient flow and functional analysis methods.
The author further provides a condition for the mass conservation of the flow with an integrable function as initial data for $3/2<p<\infty$ in a general metric space.
In contrast, our $q$-moment estimate \eqref{eq:introd_1} is a local estimate based on pure PDE and analysis method, and, while holding for $\mathbb{R}^N$, our result shows that mass conservation holds for a different range $p_c< p< 2$.

The paper is organized as follows: in Section 2, we introduce notations and definitions and auxiliary Lemmas as well as the critical parameters $p_c$, $p(N)$ and $p_N$.
Section 3, is dedicated to the main theorem for the local $q$-moment estimate \eqref{eq:introd_1}.
Section 4 addresses the mass conservation and the weak convergence.
Finally, in Section 5 we prove the Wasserstein convergence rate.

\section{Notation and preliminary results}
On $\mathbb{R}^N$ equipped with the Euclidean norm $|\cdot|$, we denote by $B_{\rho}(x)=\{y \colon |y-x| \leq \rho\}$ the closed ball centered at $x$ with radius $\rho$, and set $B_{\rho}:=B_{\rho}(0)$.
We denote by $A_{r}^{R}=\{x \colon r\leq |x|\leq R\}$ be the closed annulus centered at $0$ with radius $r$ and $R$.
Let further $\Omega\subseteq \mathbb{R}^N$ be a measurable set.
For $1\leq p\leq \infty$, we denote by $L^p\left( \Omega \right)$ the space of $p$-integrable functions with respect to the Lebesgue measure and set $L^p(\mathbb{R}^N)=L^p$.
We denote by $W^{1,p}(\Omega)$ the space of first-order Sobolev space and set $W^{1,p}(\mathbb{R}^N)=W^{1,p}$.
For a function $u$ in $W^{1,p}(\Omega)$, we denote by $\nabla u$ the weak derivative of $u$.
For $T>0$ and $S_T=\mathbb{R}^N \times (0,T]$, we denote by $L^{\infty}_{\mathrm{loc}}(S_T)$ the space of functions which belong to $L^{\infty}(K\times [s,t])$ for all intervals $[s,t]\subset (0,T]$ and compact subsets $K$ of $\mathbb{R}^N$.
For $X=L^p(\Omega)$ or $X=W^{1,p}(\Omega)$, we denote by $C([0,T];X), C((0,T);X)$ and $L^p(0,T; X)$ in a Bochner sense.
We further denote by $\mathcal{M}^+$ the set of positive Radon measures on $\mathbb{R}^N$.

\begin{definition}
    We say that a measurable function $u:S_T\to \mathbb{R}$ is a weak solution of the Cauchy problem \eqref{eq:p_laplace} with $L^1$-initial data, that is
    \begin{equation}\label{eq:initial_data_1}
        u_0 = \mu\in L^1,
    \end{equation}
    if for every bounded open set $\Omega\subset \mathbb{R}^N$, it holds
    \begin{equation}
        \begin{cases}
            \displaystyle u\in C\left((0, T); L^1\left(\Omega\right)\right)\cap L^{p-1}\left(0,T;W^{1,p-1}\left(\Omega\right)\right)\cap L^{\infty}_{loc}\left(S_T\right),\\
            |\nabla u|\in L^{\infty}_{loc}\left(S_T\right),
        \end{cases}
    \end{equation}
    and for all $0<s<t<T$ and $\phi\in C^1(\mathbb{R}^N \times [0, T])$ such that $\mathrm{supp}(\phi(\cdot,\tau))\subset B_{\rho}$ for all $\tau\in [0, T]$ for some $\rho>0$,
    \begin{equation}\label{eq:test_condtion_1}
        \int_{\mathbb{B_{\rho}}}(\phi u)(x,t)dx
        +
        \int_{s}^{t}\int_{B_{\rho}}\left(
            -u \partial_t \phi + |\nabla u|^{p-2}\nabla u \cdot \nabla \phi 
        \right)dx d\tau
        =
        \int_{B_{\rho}}(\phi u)(x,s)dx,
    \end{equation}
    and for all $R>0$,
    \begin{equation}\label{eq:initial_cond_1}
        \lim_{\tau\searrow 0}\int_{B_R}\left|u(x,t) - u_0(x)\right|dx=0.
    \end{equation}

    The measurable function $u:S_T \to \mathbb{R}^N$ is a weak solution of the Cauchy problem \eqref{eq:p_laplace} with positive Radon measure as initial data, that is,
    \begin{equation}\label{eq:initial_data_2}
        u_0=\mu\in \mathcal{M}^+,
    \end{equation}
    if condition \eqref{eq:initial_cond_1} is replaced by
    \begin{equation}
        \lim_{\tau\searrow 0}\int_{\mathbb{R}^N}\psi(x)u(x,\tau)dx=\int_{\mathbb{R}^N}\psi(x)\mu(dx)
    \end{equation}
    for all $\psi\in C_c(\mathbb{R}^N)$.
\end{definition}

For the existence and uniqueness of weak solution of \eqref{eq:p_laplace} with initial data $u_0$ in $L^1$, we refer readers to \citep{dibenedetto1990}.
In particular, if the initial data $u_0$ is in $C^{\infty}_{c}(\mathbb{R}^N)$, then the weak solution $u$ is unique and satisfies that $u\in C([0,T]; L^1(\mathbb{R}^N))$, $\partial_t u\in L^2(0,T;L^2(\mathbb{R}^N))$ and $|\nabla u|\in L^p(0,T;L^p(\mathbb{R}^N))$ and $u\in L^{\infty}(S_T)$.
In this case, the condition \eqref{eq:test_condtion_1} is equivalent to the following condition:
\begin{equation}\label{eq:test_condtion_2}
    \int_{\mathbb{R}^N}(\phi u)(x,t)dx+\int_{0}^{t}\int_{\mathbb{R}^N}\left(-u \partial_{\tau} \phi + |\nabla u|^{p-2}\nabla u \cdot \nabla \phi \right)dx d\tau
    =
    \int_{\mathbb{R}^N}\phi(x,0)u_0dx.
\end{equation}
Furthermore, $u$ is locally $\alpha$-H\"older continuous on $\Omega^{\varepsilon}_T:=\Omega\times [\varepsilon,T]$ for all bounded subsets $\Omega\subset \mathbb{R}^N$ and $\varepsilon>0$, with $\alpha\in (0,1)$ depending only on $N,p$ and supremum norm $\|u\|_{\infty,\Omega^{\varepsilon}_T}$ on $\Omega^{\varepsilon}_T$, see \cite{chen1988}.
Moreover, if initial data $u_0\geq 0$, then the weak solution $u(x,t)\geq 0$ for all $(x,t)\in S_T$.

For the existence of weak solution of \eqref{eq:p_laplace} with a positive Radon measure as initial data, we refer readers to \citep[Theorem 1]{zhao1995} and \citep[Theorem III.8.1]{dibenedetto1990}. 
However, to our knowledge, in this case the uniqueness of weak solution is unknown and depends on the choice of the definition of the solution (viscosity, entropy, distributional, etc.).

Recall that a sequence $(\mu_n)$ of positive Radon measures converges vaguely, or weakly, to $\mu$ in $\mathcal{M}^+$ if $\int_{\mathbb{R}^N}^{} \phi d\mu_n \to \int_{R^N}^{} \phi d\mu$ for every $\phi$ in $C_c(\mathbb{R}^N$, or for every $\phi$ in $C_b(\mathbb{R}^N)$, respectively.

We briefly recall two important Lemmas for the proofs that are formulated in our notations.
\begin{lemma}\label{lemma:Junning_Lemma}
    (A priori estimate, \citet{zhao1995}).
    Let $k=N(p-2)+p$ and $u$ be the weak solution of \eqref{eq:p_laplace} with $u_0\in C^{\infty}_c(\mathbb{R}^N)$ as initial data.
    Then there exists $C_1:=C_1(N,p)$ such that for all $R>0$, it holds
    \begin{equation}\label{eq:Junning_1}
        \sup_{0<\tau\leq t}\int_{B_R} \left|u(x,\tau) \right|dx
        \leq
        C_1 \int_{B_{2R}} \left|u_0 \right|dx + C_1 R^{-\frac{k}{2-p}}t^{\frac{1}{2-p}},
    \end{equation}
    and for any $R_0$, there exists $C_2:=C_2(N,p,R_0)$ such that for all $R>R_0$, it holds
    \begin{equation}\label{eq:Junning_2}
        \sup_{x\in B_R} \left|u(x,t) \right|
        \leq
        C_2 \left(
            t^{-\frac{N}{k}}\left(
                \int_{B_{4R}}|u_0|dx
            \right)^{\frac{p}{k}} + R^{-\frac{p}{2-p}}t^{\frac{1}{2-p}}
        \right).
    \end{equation}
\end{lemma}
\begin{lemma}\label{lemma:Mariano_lemma}(\citep[Lemma 3.1]{mariano1983})
    Let $f\colon [r_0, r_1]\to [0, \infty)$ be a bounded function with $0\leq r_0< r_1$.
    Suppose that for any $\sigma,\sigma^\prime$ with $r_0\leq \sigma<\sigma^\prime\leq r_1$, we have
    \begin{equation*}
        f(\sigma)
        \leq
        \left[ A\left(\sigma^\prime -\sigma \right)^{-a} + B \right] +\theta f\left(\sigma'\right),
    \end{equation*}
    where $A,B,a,\theta$ are nonnegative constants with $0\leq \theta<1$.
    Then for any $\lambda,\lambda^\prime$ with $r_0\leq \lambda <\lambda^\prime\leq r_1$, it holds
    \begin{equation*}
        f(\lambda)
        \leq
        C\left(a,\theta\right) \left[A\left(\lambda^\prime - \lambda\right)^{-a} + B \right],
    \end{equation*}
    for some positive constant $C(a,\theta)$ depending only on $a$ and $\theta$.
\end{lemma}

Finally, recall the following critical values
\begin{equation}\label{eq:polynomial}
    \begin{split}
        p_c & := \frac{2N}{N+1}, \\
        p(N) &: = (N+2)p^2 - (3N + 3) p +2N, \\
        p_N &:= \text{largest root of }p(N).
    \end{split}
\end{equation}
Note that since $N\geq 2$, it holds
\begin{equation*}
    \frac{4}{3}\leq p_c < p_N <2.
\end{equation*}

\section{Moments estimates}\label{sec:model}
In this section, we derive $q$-moment estimates for a weak solution of \eqref{eq:p_laplace} with finite positive Radon measure as initial data.
We denote by $q$ the H\"older conjugate number of $p$, that is, $\frac{1}{p}+\frac{1}{q}=1$, and denote by $k$ the constant $N(p-2)+p$.
Given a $\mu\in \mathcal{M}^+$, we say that $\mu$ has finite $q$-moment if $\int_{\mathbb{R}^N}|x|^q d\mu<\infty$.
We begin to state our main result in this section.

\begin{theorem}\label{theorem:moment_estimate}
    ($q$-moment estimates.)
    Let $p_c<p<2$ and $\mu$ be in $\mathcal{M}^+$.
    Then for any $0<r<R$ and $0<t\leq T$, there exists a weak solution $u$ of \eqref{eq:p_laplace} with $\mu$ as initial data satisfying
    \begin{multline}\label{eq:thm_result_ineq}
        \sup_{0< \tau \leq t} \int_{r\leq|x|\leq R} \left| x \right|^q u(x,\tau)dx
        \leq
        C \int_{\frac{1}{2}r\leq|x|\leq 2R} |x|^q d\mu\\
        +
        C t^{\frac{1}{2-p}}\left(
            r^{-\frac{p(N)}{(p-1)(2-p)}} + R^{-\frac{p(N)}{(p-1)(2-p)}}
        \right),
    \end{multline}
    for some constant $C:=C(N,p)$ depending only on $N$ and $p$.
    Furthermore, if $\mu$ in $\mathcal{M}^+$ has finite total mass and finite $q$-moment, and $p(N)>0$, then for any $0<t\leq T$, the family $\{u(x,\tau)dx\colon 0\leq \tau\leq t\}$ has uniformly bounded $q$-moment on $[0, t]$ and that
    \begin{equation}\label{eq:q_moment_uniformly_integrable}
        \lim_{r\rightarrow \infty}\sup_{0<\tau\leq t}\int_{|x|\geq r}|x|^q u(x,\tau)dx
        =
        0.
    \end{equation}
\end{theorem}
\begin{remark}
    Let $\{X_{\tau}\colon 0\leq \tau\leq t\}$ be a family of $\mathbb{R}^N$-valued random variables such that the law of $X_\tau$ is $u(x,\tau)dx$.
    Then equality \eqref{eq:q_moment_uniformly_integrable} is equivalent to say that $\{|X_{\tau}|^q\colon 0\leq \tau\leq t\}$ is uniformly integrable for any $0<t\leq T$.
\end{remark}

\begin{remark}
    Note that $(p_N,2)$ is a sharp range for $p$ in the following sense: let $U$ be the Barenblatt source-type solution, that is,
    \begin{equation}
        U(x,t)
        =t^{-\frac{N}{k}}F(t^{-\frac{1}{k}}x),
    \end{equation}
    where $F(\xi)=(C+\gamma|\xi|^{\frac{p}{p-1}})^{\frac{p-1}{p-2}}$ for $\gamma=(p-1)k/(p(2-p))$ and some positive constant $C$.
    It is well-known that $w$ is a weak solution of \eqref{eq:p_laplace} with initial data $\mu=M\delta_0$ where $M=\| w(t) \|_{L^1}$.
    Computations show that $U$ has finite $q$-moment if and only if $p_N<p<2$.
\end{remark}

To show Theorem \ref{theorem:moment_estimate}, we first establish the following auxiliary lemma.
\begin{lemma}\label{lemma:aux_lemma_1}
    Let $p\in (p_c,2)$ and $\xi\in C^1_c(\mathbb{R}^N)$ be a smooth function such that $\mathrm{supp}(\xi)\subset \Omega$ for some bounded domain $\Omega$.
    Let $u$ be the positive weak solution of \eqref{eq:p_laplace} with initial data $u_0\in C^{\infty}_c(\mathbb{R}^N),u_0\geq 0$.
    Then there exists a positive constant $C:=C(N,p)$ such that for any $\varepsilon>0$, 
    \begin{multline*}
        \int_{0}^{t}\int_{\Omega}\xi^p \tau^{\beta q}|\nabla u|^p (u+\varepsilon)^{-\alpha q}dxd\tau
        \leq
        C \int_{0}^{t}\int_{\Omega}\tau^{\beta q}(u+\varepsilon)^{p-\alpha q}|\nabla \xi|^{p}dxd\tau\\
        +
        C \int_{\Omega}t^{\beta q}(u+\varepsilon)^{2-\alpha q}\xi^pdx,
    \end{multline*}
    where $\beta=\alpha/2$ and $\alpha=1/p$ if $p>1/(p-1)$ and $\alpha=p-1$ if $p\leq 1/(p-1)$.
\end{lemma}
\begin{proof}
    Let $\varepsilon>0$ be fixed and $u_{\varepsilon}=u+\varepsilon$ and $\psi(x,\tau)=\tau^{\beta q}u_{\varepsilon}^{1-\alpha q}\xi^p$.
    Note that $u_{\varepsilon}^{1-\alpha p}$ is bounded since $-1<1-\alpha p<0$.

    On the one hand, by the regularity of $u$ on $(0,T)$ it follows that
    \begin{equation*}
        \partial_\tau \left( \tau^{\beta q} u_\varepsilon^{2-\alpha q}\xi^p \right) 
        = 
        \left( 2 - \alpha q \right)\psi \partial_{\tau}u 
        + 
        \beta q \tau^{\beta q-1} u_{\varepsilon}^{2-\alpha q} \xi^p
        \geq 
        \left(2-\alpha q \right)\psi \partial_{\tau}u.
    \end{equation*}
    Since $0<2-\alpha q<1$, with $\partial_{\tau}u$ in $L^2(0,T;L^2(\mathbb{R}^N))$ and $u$ in $C([0, T]; L^1(\mathbb{R}^N))$ and boundedness of $\psi$, we can integrate over $\Omega \times (0,t)$ on both sides and get
    \begin{equation}\label{eq:tmp_1}
        \int_{0}^{t}\int_{\Omega}^{} \psi \partial_\tau u dx d\tau
        \leq 
        \frac{1}{2-\alpha q}\int_{\Omega} t^{\beta q}u_{\varepsilon}^{2-\alpha q}(x,t)\xi^p dx.
    \end{equation}
    By the regularity of $u$ it follows that
    \begin{equation}\label{eq:tmp_1_1}
        |\nabla u|^{p-2}\nabla u \nabla \psi
        =
        p \xi^{p-1}|\nabla u|^{p-2}\tau^{\beta q}u_{\varepsilon}^{1-\alpha q}\nabla u \nabla \xi
        -
        \left(\alpha q-1 \right)|\nabla u|^{p} \xi^{p}\tau^{\beta q}u_{\varepsilon}^{-\alpha q}.
    \end{equation}
    For the first term on the right hand side of \eqref{eq:tmp_1_1}, applying Young's inequality $a\cdot b\leq \varepsilon |a|^q + \varepsilon^{-\frac{1}{q-1}}|b|^p$ for $\varepsilon = (\alpha q-1)/(2p)$ and 
    \begin{equation*}
        a = \xi^{p-1}|\nabla u|^{p-2} \tau^{\beta} u_{\varepsilon}^{-\alpha}\nabla u \quad 
        \text{and} \quad 
        b = \tau^{\beta (q-1)}u_{\varepsilon}^{1-\alpha (q-1)} \nabla \xi.
    \end{equation*}
    Equation \eqref{eq:tmp_1_1} yields
    \begin{multline}\label{eq:tmp_2_2}
        |\nabla u|^{p-2}\nabla u \nabla \psi
        \leq
        -\frac{1}{2}\left(\alpha q-1\right)\xi^p |\nabla u|^p \tau^{\beta q}u_{\varepsilon}^{-\alpha q}\\
        +
        p\left(\frac{2p}{\alpha q-1}\right)^{\frac{1}{q-1}} \tau^{\beta q}u_{\varepsilon}^{p-\alpha q}|\nabla \xi|^p.
    \end{multline}
    Since $|\nabla u|$ is in $L^p(0,T; L^p(\mathbb{R}^N))$ and $u_{\varepsilon}^{-\alpha q}$ is bounded and $0<p-\alpha q<1$, it follows that
    \begin{equation*}
        \int_{0}^{t}\int_{\Omega}\xi^p |\nabla u|^p \tau^{\beta q}u_{\varepsilon}^{-\alpha q}dx d\tau < \infty
        \quad \text{and} \quad
        \int_{0}^{t}\int_{\Omega} \tau^{\beta q}u_{\varepsilon}^{p-\alpha q}|\nabla \xi|^p dx d\tau < \infty.
    \end{equation*}
    Hence, integrating \eqref{eq:tmp_2_2} over $\Omega \times (0,t)$ yields
    \begin{multline}\label{eq:tmp_2}
        \int_{0}^{t}\int_{\Omega}|\nabla u|^{p-2}\nabla u \nabla \psi dxd\tau
        \leq
        -\frac{1}{2}\left(\alpha q-1\right)\int_{0}^{t}\int_{\Omega}\xi^p |\nabla u|^p \tau^{\beta q}u_{\varepsilon}^{-\alpha q}dx d\tau\\
        +
        p\left(\frac{2p}{\alpha q-1}\right)^{\frac{1}{q-1}}\int_{0}^{t}\int_{\Omega} \tau^{\beta q}u_{\varepsilon}^{p-\alpha q}|\nabla \xi|^p dx d\tau.
    \end{multline}

    On the other hand, multiplying \eqref{eq:p_laplace} by $\psi$ and integrating by part, it follows that
    \begin{equation*}
        \int_{0}^{t}\int_{\Omega} \psi \partial_{\tau} u dx d\tau
        +
        \int_{0}^{t}\int_{\Omega} |\nabla u|^{p-2}\nabla u \nabla \psi dx d\tau
        =
        0.
    \end{equation*}
    Hence, adding \eqref{eq:tmp_1} and \eqref{eq:tmp_2} together yields
    \begin{multline*}
        \int_{0}^{t}\int_{\Omega} \xi^p |\nabla u|^{p} \tau^{\beta q}u_{\varepsilon}^{-\alpha q}dx d\tau
        \leq
        \left(\frac{2p}{\alpha q -1}\right)^p \int_{0}^{t}\int_{\Omega}\tau^{\beta q}u_{\varepsilon}^{p-\alpha q}|\nabla \xi|^{p}dxd\tau\\
        +
        \frac{2}{\left(\alpha q-1\right)\left(2-\alpha q\right)}\int_{\Omega} t^{\beta q}u_{\varepsilon}^{2-\alpha q}(x,t)\xi^p dx d\tau.
    \end{multline*}
    Taking
    \begin{equation*}
        C(N, p)=\max\left\{\left(\frac{2p}{\alpha q - 1}\right)^{-p}, \frac{2}{\left(\alpha q-1\right)\left(2-\alpha q\right)}\right\}  
    \end{equation*}
    yields the result.
\end{proof}

We address now the proof of Theorem \ref{theorem:moment_estimate}.
In the following proof, we denote by $C$ a generic positive constant depending only on $p$ and $N$.
\begin{proof}[Proof of Theorem \ref{theorem:moment_estimate}]
    In a first step, we show inequality \eqref{eq:thm_result_ineq} in the case where $u$ is the positive weak solution of \eqref{eq:p_laplace} with initial data $u_0$ in $C^{\infty}_{c}(\mathbb{R}^N)$ and $u_0\geq 0$.
    Let $R>r>0$ and $1\leq \sigma<\sigma^{\prime}\leq 2$. 
    Let $A_{\sigma}, A_{\sigma^{\prime}}$ denote closed annulus $A_{r/\sigma}^{R\sigma}$ and $A_{r/\sigma^{\prime}}^{R\sigma^{\prime}}$ respectively.
    Let $\xi$ be a smooth cut-off function such that $\xi=1$ in $A_{\sigma}$ and $\xi=0$ in $A^c_{\sigma^{\prime}}$ and $|\nabla \xi|\leq \frac{\sigma^{\prime}\sigma}{(\sigma'-\sigma)r}$ in $A^r:=A_{r/\sigma^{\prime}}^{r/\sigma}$ and $|\nabla \xi|\leq \frac{1}{(\sigma'-\sigma)R}$ in $A^R:=A_{R\sigma}^{R\sigma^{\prime}}$.
    Let $\phi(x,\tau)=|x|^q \xi^p \chi_{[0,T]}(\tau)$ be the test function where $\chi_{[0,T]}$ is indicator function.
    Then for any $0<t\leq T$, by \eqref{eq:test_condtion_2}, it follows that
    \begin{multline}\label{eq:thm_1_1}
        \int_{A_{\sigma}}|x|^q u(x,t)dx
        \leq
        \int_{A_{\sigma'}}|x|^q u_0 dx
        +
        p \int_{0}^{t}\int_{A_{\sigma'}}|\nabla u|^{p-1}\xi^{p-1} |x|^q |\nabla \xi| dx d\tau\\
        +
        q \int_{0}^{t}\int_{A_{\sigma'}}|\nabla u|^{p-1}\xi^p |x|^{\frac{1}{p-1}}dxd\tau.
    \end{multline}
    For the second term on the right hand side of \eqref{eq:thm_1_1}, since $|\nabla \xi|=0$ in $A_{\sigma}$ and $|\nabla \xi|\leq \frac{\sigma'\sigma}{(\sigma'-\sigma)r}$ and $|\nabla\xi|\leq \frac{1}{(\sigma'-\sigma)R}$ in $A^r$ and $A^R$ respectively, and note that $\sigma, \sigma'\leq 2$, it follows that
    \begin{multline}\label{eq:thm_1_2}
        \int_{0}^{t}\int_{A_{\sigma'}} |\nabla u|^{p-1}\xi^{p-1}|x|^q |\nabla \xi|dxd\tau
        =
        \int_{0}^{t}\int_{A^r \cup A^R} |\nabla u|^{p-1}\xi^{p-1}|x|^{\frac{1}{p-1}} \left(|x| |\nabla \xi| \right) dxd\tau\\
        \leq
        \frac{4}{\sigma^{\prime}-\sigma}\int_{0}^{t}\int_{A_{\sigma'}}|\nabla u|^{p-1}\xi^{p-1}|x|^{\frac{1}{p-1}}dxd\tau.
    \end{multline}
    Plugging \eqref{eq:thm_1_2} into \eqref{eq:thm_1_1} and use the fact that $\xi^{p}\leq \xi^{p-1}$ and $1\leq 1/(\sigma^{\prime}-\sigma)$ yields
    \begin{equation}\label{eq:thm_1_3}
        \int_{A_{\sigma}}|x|^q u(x,t)dx
        \leq
        \int_{A_{\sigma'}}|x|^q u_0 dx
        +
        \frac{C}{\sigma'-\sigma}\int_{0}^{t}\int_{A_{\sigma'}}|\nabla u|^{p-1}\xi^{p-1}|x|^{\frac{1}{p-1}}dxd\tau.
    \end{equation}
    Let $u_{\varepsilon}:=u+\varepsilon$ for $\varepsilon>0$.
    We consider the second term on the right hand side of \eqref{eq:thm_1_3} for different cases of $p$:
    \begin{enumerate}[label=\textit{Case \arabic*:}, fullwidth]
        \item If $p>1/(p-1)$, then by H\"older's inequality, it follows that
            \begin{multline}\label{eq:thm_1_4}
                \int_{0}^{t}\int_{A_{\sigma'}}|\nabla u|^{p-1}\xi^{p-1}|x|^{\frac{1}{p-1}}dxd\tau\\
                \leq
                \left(
                    \int_{0}^{t}\int_{A_{\sigma'}}\tau^{\beta q}|\nabla u|^{p} \xi^{p} u_{\varepsilon}^{-\alpha q}dxd\tau
                \right)^{\frac{1}{q}}
                \left(
                    \int_{0}^{t}\int_{A_{\sigma'}}\tau^{-\beta p}|x|^{q}u_{\varepsilon}^{\alpha p} dxd\tau
                \right)^{\frac{1}{p}},
            \end{multline}
            where $\alpha=1/p$ and $\beta=1/(2p)$.

            For the first term on the right hand of \eqref{eq:thm_1_4}, by Lemma \ref{lemma:aux_lemma_1} and upper bounds of $|\nabla \xi|$ in $A^r$ and $A^R$, it follows that 
            \begin{multline}\label{eq:thm_1_5}
                \int_{0}^{t}\int_{A_{\sigma'}}\tau^{\beta q}|\nabla u|^{p} \xi^{p} u_{\varepsilon}^{-\alpha q}dxd\tau
                \leq
                \frac{C}{[(\sigma'-\sigma)r]^p} \int_{0}^{t}\int_{A^r}\tau^{\beta q}u_{\varepsilon}^{p-\alpha q}dxd\tau\\
                +
                \frac{C}{[(\sigma'-\sigma)R]^p} \int_{0}^{t}\int_{A^R}\tau^{\beta q}u_{\varepsilon}^{p-\alpha q}dxd\tau
                +
                C \int_{A_{\sigma'}}t^{\beta q}u_{\varepsilon}^{2-\alpha q} dx.
            \end{multline}
            For the first term on the right hand side of \eqref{eq:thm_1_5}, since  $0<p-1/(p-1)<1$, H\"older's inequality yields
            \begin{multline}\label{eq:thm_1_6}
                \int_{0}^{t}\int_{A^r}\tau^{\frac{q}{2p}}u_{\varepsilon}^{p-\frac{1}{p-1}}dxd\tau
                =
                \int_{0}^{t}\int_{A^r} \left( \tau^{\frac{p}{2}}|x|^{-q\left(p-\frac{1}{p-1}\right)}\right)
                \left( |x|^{q}\tau^{-\frac{1}{2}}u_{\varepsilon}\right)^{p-\frac{1}{p-1}}dxd\tau\\
                \leq
                \left(
                    \int_{0}^{t}\int_{A^r} \tau^{\frac{p}{2(q-p)}}|x|^{-\frac{q}{q-p}\left(p-\frac{1}{p-1}\right)} dxd\tau
                \right)^{1-p+\frac{1}{p-1}}
                \left(
                    \int_{0}^{t}\int_{A^r}\tau^{-\frac{1}{2}}|x|^{q} u_{\varepsilon} dxd\tau
                \right)^{p-\frac{1}{p-1}}\\
                \leq
                C t^{q-\frac{q}{2p}} r^{N(q-p)-q\left(p-\frac{1}{p-1}\right)}
                \left(
                    \sup_{0<\tau\leq t}\int_{A_{\sigma'}}|x|^{q} u_{\varepsilon} dx
                \right)^{p-\frac{1}{p-1}}.
            \end{multline}
            The second term on the right hand side of \eqref{eq:thm_1_5} is the same as the first term, replacing $r$ by $R$.
            As for the third term on the right hand side of \eqref{eq:thm_1_5}, H\"older's inequality yields
            \begin{multline}\label{eq:thm_1_7}
                \int_{A_{\sigma'}}t^{\frac{q}{2p}}u_{\varepsilon}^{2-\frac{1}{p-1}}dx
                \leq
                t^{\frac{q}{2p}}
                \left(
                    \int_{A_{\sigma'}}|x|^{-\frac{p(2p-3)}{(p-1)(2-p)}}dx
                \right)^{\frac{1}{p-1}-1}
                \left(
                    \int_{A_{\sigma'}}|x|^q u_{\varepsilon} dx
                \right)^{2-\frac{1}{p-1}}\\
                \leq
                C t^{\frac{q}{2p}}\left(\max\left\{r^{-\frac{p(N)}{(p-1)(2-p)}},R^{-\frac{p(N)}{(p-1)(2-p)}}\right\}\right)^{\frac{2-p}{p-1}}
                \left(
                    \sup_{0<\tau\leq t}\int_{A_{\sigma'}}|x|^q u_{\varepsilon} dx
                \right)^{\frac{2p-3}{p-1}}.
            \end{multline}

            We now turn to the second term on right hand side of \eqref{eq:thm_1_4} for which holds
            \begin{equation}\label{eq:thm_1_7_1}
                \left(\int_{0}^{t}\int_{A_{\sigma^{\prime}}}\tau^{-\frac{1}{2}}|x|^q u_{\varepsilon}dx d\tau\right)^{\frac{1}{p}}
                \leq
                C t^{\frac{1}{2p}}\left(\sup_{0<\tau\leq t}\int_{A_{\sigma^{\prime}}}|x|^q u_{\varepsilon}dx\right)^{\frac{1}{p}}.
            \end{equation}
            Plugging \eqref{eq:thm_1_6}, \eqref{eq:thm_1_7} into \eqref{eq:thm_1_5}, and then applying inequality $(a+b)^{1/q}\leq a^{1/q}+b^{1/q}$ for $a,b>0$ to \eqref{eq:thm_1_5}, and plugging it into \eqref{eq:thm_1_4} together with \eqref{eq:thm_1_7_1} and taking $\varepsilon \searrow 0$, it follows that
            \begin{multline}\label{eq:thm_1_8}
                \frac{C}{\sigma^{\prime}-\sigma}\int_{0}^{t}\int_{A_{\sigma'}}|\nabla u|^{p-1}\xi^{p-1}|x|^{\frac{1}{p-1}}dxd\tau\\
                \leq
                C (\sigma'-\sigma)^{-p} t \left( r^{1-k-\left(p-\frac{1}{p-1}\right)} + R^{1-k-\left(p-\frac{1}{p-1}\right)}\right)
                M_q(\sigma^{\prime})^{p-1}\\
                + 
                C (\sigma'-\sigma)^{-1} t^{\frac{1}{p}}\left(\max\left\{r^{-\frac{p(N)}{(p-1)(2-p)}},R^{-\frac{p(N)}{(p-1)(2-p)}}\right\}\right)^{\frac{2-p}{p}}
                M_q(\sigma^{\prime})^{\frac{2(p-1)}{p}},
            \end{multline}
            where $M_q(\sigma^{\prime})=\sup_{0<\tau\leq t}\int_{A_{\sigma^{\prime}}}|x|^q u dx$.
            Applying Young's inequaliy $ab\leq \varepsilon b^{1/(p-1)} + C_{\varepsilon}a^{1/(2-p)}$ and $ab\leq \varepsilon b^{p/(2p-2)}+ C_{\varepsilon}a^{p/(2-p)}$ to the first and second terms of the right hand side of \eqref{eq:thm_1_8}, respectively, whereby $\varepsilon=1/4$, and using identity $1-k-p+1/(p-1)=-p(N)/(p-1)$, it follows that
            \begin{multline}\label{eq:thm_case_1}
                \frac{C}{\sigma^{\prime}-\sigma}\int_{0}^{t}\int_{A_{\sigma^{\prime}}}|\nabla u|^{p-1}\xi^{p-1}|x|^{\frac{1}{p-1}} dx d\tau
                \leq
                \frac{1}{2}\sup_{0<\tau\leq t}\int_{A_{\sigma^{\prime}}}|x|^q u dx\\
                +
                C (\sigma'-\sigma)^{-\frac{p}{2-p}}t^{\frac{1}{2-p}}
                \left(
                    r^{-\frac{p(N)}{(p-1)(2-p)}}
                    +
                    R^{-\frac{p(N)}{(p-1)(2-p)}}
                \right).
            \end{multline}

        \item If $p\leq 1/(p-1)$, by H\"older's inequality, it follows that
            \begin{multline}\label{eq:thm_3_1}
                \int_{0}^{t}\int_{A_{\sigma^{\prime}}} |\nabla u|^{p-1} \xi^{p-1}|x|^{\frac{1}{p-1}} dx d\tau\\
                \leq
                \left( \int_{0}^{t}\int_{A_{\sigma^{\prime}}}\tau^{\beta q}|\nabla u|^p u_{\varepsilon}^{-\alpha q}\xi^{p}|x|^{\left(\frac{1}{p-1}-p\right)q} dx d\tau \right)^{\frac{1}{q}}
                \left( \int_{0}^{t}\int_{A_{\sigma^{\prime}}}\tau^{-\beta p}|x|^{p^2}u_{\varepsilon}^{\alpha p} dx d\tau\right)^{\frac{1}{p}},
            \end{multline}
            where $\alpha=p-1$ and $\beta=(p-1)/2$.

            For the first term on the right hand side of \eqref{eq:thm_3_1}, taking $\tilde{\xi}=\xi |x|^{(1/(p-1)-p)/(p-1)}$ and applying Lemma \ref{lemma:aux_lemma_1}, yields
            \begin{multline}\label{eq:thm_3_2}
                \int_{0}^{t}\int_{A_{\sigma^{\prime}}} \tau^{\beta q}|\nabla u|^p u_{\varepsilon}^{-\alpha q}\left(\xi |x|^{\left(\frac{1}{p-1}-p\right)\frac{1}{p-1}}\right)^{p} dx d\tau
                \leq
                C\int_{0}^{t}\int_{A_{\sigma^{\prime}}}\tau^{\beta q}u_{\varepsilon}^{p-\alpha q} |\nabla \tilde{\xi}|^p dx d\tau\\
                +
                C \int_{A_{\sigma^{\prime}}} t^{\beta q} u_{\varepsilon}^{2- \alpha q} \tilde{\xi}^p dx.
            \end{multline}

            For the first term on the right hand side of \eqref{eq:thm_3_2}, note that we have
            \begin{equation*}
                \nabla \tilde{\xi}
                =
                |x|^{\left(\frac{1}{p-1}-p\right)\frac{1}{p-1}}\nabla \xi
                +
                \xi \left(\frac{1}{p-1}-p\right)\frac{1}{p-1} |x|^{\left(\frac{1}{p-1}-p\right)\frac{1}{p-1}-2} x.
            \end{equation*}
            Hence, applying inequality $|a+b|^p\leq 2^{p-1}(|a|^p + |b|^p)$ and upper bounds of $|\nabla \xi|$ in $A^r$ and $A^R$ and $|\xi|\leq 1$, it follows that
            \begin{multline}\label{eq:thm_3_3}
                \int_{0}^{t}\int_{A_{\sigma^{\prime}}}\tau^{\frac{p}{2}}|\nabla \tilde{\xi}|^p dx d\tau
                \leq
                C t^{\frac{p}{2}+1}
                \left\{ 
                    \int_{A_{\sigma^{\prime}}}|x|^{\left(\frac{1}{p-1}-p\right)q} |\nabla \xi|^p
                    +
                    |x|^{\left(\frac{1}{p-1}-p\right)q-p}dx
                \right\}\\
                \leq
                C t^{\frac{p}{2}+1}\left(\sigma^{\prime}-\sigma\right)^{-p}
                \left(
                    r^{-p}\int_{A^r}|x|^{\left(\frac{1}{p-1}-p\right)q}dx
                    +
                    R^{-p}\int_{A^R}|x|^{\left(\frac{1}{p-1}-p\right)q}dx
                \right)\\
                +
                C t^{\frac{p}{2}+1}\int_{A_{\sigma^{\prime}}} |x|^{\left(\frac{1}{p-1}-p\right)q-p}dx\\
                \leq
                C t^{\frac{p}{2}+1}
                \left(\sigma^{\prime}-\sigma\right)^{-p}\left(r^{N+\left(\frac{1}{p-1}-p\right)q - p} + R^{N+\left(\frac{1}{p-1}-p\right)q - p}\right)\\
                +
                C t^{\frac{p}{2}+1} \max\left\{r^{N+\left(\frac{1}{p-1}-p\right)q - p}, R^{N+\left(\frac{1}{p-1}-p\right)q - p} \right\}.
            \end{multline}
            By the fact that $N+(1/(p-1)-p)q-p\geq 0$ when $N\geq 2$ and $r\leq R$ and $1\leq (\sigma^{\prime}-\sigma)^{-p}$, it follows from \eqref{eq:thm_3_3} that
            \begin{equation}\label{eq:thm_3_4}
                \int_{0}^{t}\int_{A_{\sigma^{\prime}}} \tau^{\frac{p}{2}} |\nabla \tilde{\xi}|^p dx d\tau
                \leq
                C t^{\frac{p}{2}+1}\left(\sigma^{\prime}-\sigma\right)^{-p}R^{N+\left(\frac{1}{p-1}-p\right)q - p}.
            \end{equation}

            As for the second term on the right hand side of \eqref{eq:thm_3_2}, note that $|\tilde{\xi}|^p\leq |x|^{(1/(p-1)-p)q}$, which by H\"older's inequality yields
            \begin{multline}\label{eq:thm_3_5}
                \int_{A_{\sigma^{\prime}}} t^{\frac{p}{2}} u_{\varepsilon}^{2-p}\tilde{\xi}^p dx
                \leq
                t^{\frac{p}{2}}\left(
                    \int_{A_{\sigma^{\prime}}} |x|^{\left(\frac{1}{p-1}-p\right)q -(2-p)q} u_{\varepsilon}^{2-p}|x|^{(2-p)q} dx
                \right)\\
                \leq
                t^{\frac{p}{2}}
                \left(
                    \int_{A_{\sigma^{\prime}}}|x|^{\left(\frac{1}{p-1}-2\right)\frac{q}{p-1}} dx
                \right)^{p-1}
                \left(
                    \int_{A_{\sigma^{\prime}}} |x|^q u_{\varepsilon} dx
                \right)^{2-p}\\
                \leq
                C t^{\frac{p}{2}} \max\left\{r^{N+\frac{p(3-2p)}{(p-1)^3}}, R^{N+\frac{p(3-2p)}{(p-1)^3}} \right\}^{p-1}\left( \sup_{0<\tau\leq t}\int_{A_{\sigma^{\prime}}}|x|^q u_{\varepsilon}dx\right)^{2-p}.
            \end{multline}
            Note that if $p\leq 3/2$, it is obvious that $N+p(3-2p)/(p-1)^3 \geq 0$.
            If $p>3/2$, since $N\geq 2$, then it holds that
            \begin{equation*}
                N(p-1)^3 + p(3-2p)
                \geq
                \frac{2}{2}(p-1)^2 + p(3-2p)
                =
                -p^2+p+1
                \geq
                0.
            \end{equation*}
            Taking $\varepsilon\searrow 0$ on the right hand side of \eqref{eq:thm_3_5}, yields
            \begin{equation}\label{eq:thm_3_6}
                \int_{A_{\sigma^{\prime}}}t^{\frac{p}{2}}u_{\varepsilon}^{2-p}\tilde{\xi}^p dx
                \leq
                C t^{\frac{p}{2}} R^{\left(N+\frac{p(3-2p)}{(p-1)^3}\right)(p-1)}M_q(\sigma^{\prime})^{2-p}.
            \end{equation}
            Plugging \eqref{eq:thm_3_4} and \eqref{eq:thm_3_6} into \eqref{eq:thm_3_2} and applying inequality $(a+b)^{1/q}\leq a^{1/q} + b^{1/q}$ yields
            \begin{multline}\label{eq:thm_3_7}
                \left(
                    \int_{0}^{t}\int_{A_{\sigma^{\prime}}} \tau^{\beta q}|\nabla u|^{p} u_{\varepsilon}^{-\alpha q}\xi^p |x|^{\left(\frac{1}{p-1}-p\right)q} dx d\tau
                \right)^{\frac{1}{q}}\\
                \leq
                C t^{\frac{p+2}{2q}}\left(\sigma^{\prime}-\sigma\right)^{1-p} R^{\frac{N}{q}+\frac{1}{p-1}-p-\frac{p}{q}}
                +
                C t^{\frac{p}{2q}} R^{\left(N + \frac{p(3-2p)}{(p-1)^3}\right)\frac{(p-1)}{q}} M_q(\sigma^{\prime})^{\frac{2-p}{q}}.
            \end{multline}

            As for the second term on the right hand side of \eqref{eq:thm_3_1}, H\"older's inequality and then letting $\varepsilon \searrow 0$, yields
            \begin{multline}\label{eq:thm_3_8}
                \left(
                    \int_{0}^{t}\int_{A_{\sigma^{\prime}}}\tau^{-\frac{p(p-1)}{2}} |x|^{p^2} u_{\varepsilon}^{p(p-1)} dx d\tau
                \right)^{\frac{1}{p}}\\
                \leq
                \left\{
                    \left(\int_{0}^{t}\int_{A_{\sigma^{\prime}}} dxd\tau \right)^{1-p^2+p}
                    \left( \int_{0}^{t}\int_{A_{\sigma^{\prime}}} \tau^{-\frac{1}{2}}|x|^{q} u_{\varepsilon} dx dx \right)^{p^2-p}
                \right\}^{\frac{1}{p}}\\
                \leq
                C t^{\frac{1}{p}-\frac{p-1}{2}} R^{\frac{N(-p^2+p+1)}{p}} M_q(\sigma^{\prime})^{p-1}.
            \end{multline}
            Plugging \eqref{eq:thm_3_7} and \eqref{eq:thm_3_8} into \eqref{eq:thm_3_1}, it follows that
            \begin{multline}\label{eq:thm_3_9}
                \frac{C}{\sigma^{\prime}-\sigma}\int_{0}^{t}\int_{A_{\sigma^{\prime}}}|\nabla u|^{p-1}\xi^{p-1}|x|^{\frac{1}{p-1}} dx d\tau
                \leq
                C \left(\sigma^{\prime}-\sigma\right)^{-p} t R^{N(2-p)-\frac{p(2p-3)}{p-1}}M_q(\sigma^{\prime})^{p-1}\\
                +
                C \left(\sigma^{\prime}-\sigma\right)^{-1}t^{\frac{1}{p}} R^{\frac{N(2-p)}{p}+\frac{3-2p}{p-1}} M_q(\sigma^{\prime})^{\frac{2(p-1)}{p}}.
            \end{multline}
            Applying Young's inequality $ab\leq \varepsilon b^{1/(p-1)} + C_{\varepsilon}a^{1/(2-p)}$ as well as $ab\leq \varepsilon b^{p/(2p-2)}+ C_{\varepsilon}a^{p/(2-p)}$ to the first and second terms on the right hand side of \eqref{eq:thm_3_9} respectively, where $\varepsilon=1/4$, it follows that
            \begin{multline}\label{eq:thm_case_2}
                \frac{C}{\sigma^{\prime}-\sigma}\int_{0}^{t}\int_{A_{\sigma^{\prime}}}|\nabla u|^{p-1}\xi^{p-1}|x|^{\frac{1}{p-1}} dx d\tau
                \leq
                \frac{1}{2}\sup_{0<\tau\leq t}\int_{A_{\sigma^{\prime}}}|x|^q u dx\\
                +
                C \left(\sigma^{\prime}-\sigma\right)^{\frac{p}{2-p}}t^{\frac{1}{2-p}}R^{-\frac{p(N)}{(p-1)(2-p)}}.
            \end{multline}
    \end{enumerate}   

    Together with \eqref{eq:thm_case_1} and \eqref{eq:thm_case_2}, we deduce from \eqref{eq:thm_1_3} that
    \begin{multline}\label{eq:thm_4_1}
        \int_{A_{\sigma}}|x|^q u(x,t)dx
        \leq
        \int_{A_{\sigma^{\prime}}}|x|^q u_0 dx
        +
        \frac{1}{2}\sup_{0<\tau\leq t}\int_{A_{\sigma^{\prime}}}|x|^q u dx\\
        +
        C \left(\sigma^{\prime}-\sigma\right)^{-\frac{p}{2-p}}t^{\frac{1}{2-p}}\left( r^{-\frac{p(N)}{(p-1)(2-p)}}+R^{-\frac{p(N)}{(p-1)(2-p)}} \right).
    \end{multline}
    Inequality \eqref{eq:thm_4_1} holds by replacing $t$ by any $\tau$ in $(0,t]$ implying that
    \begin{multline}
        \sup_{0<\tau\leq t}\int_{A_{\sigma}}|x|^q u dx
        \leq
        \int_{A_{\sigma^{\prime}}}|x|^q u_0 dx + \frac{1}{2}\sup_{0<\tau\leq t}\int_{A_{\sigma^{\prime}}}|x|^q u dx\\
        +
        C \left(\sigma^{\prime}-\sigma\right)^{-\frac{p}{2-p}}t^{\frac{1}{2-p}}\left( r^{-\frac{p(N)}{(p-1)(2-p)}}+R^{-\frac{p(N)}{(p-1)(2-p)}} \right).
    \end{multline} 
    By Lemma \ref{lemma:Junning_Lemma}, it holds that
    \begin{equation*}
        \sup_{\sigma\in [1,2]}M_q(\sigma)
        \leq
        (2R)^q \sup_{0<\tau\leq t}\int_{B_{2R}}u dx
        \leq
        C (2R)^q\left(
            \int_{B_{2R}}u_0dx + R^{-\frac{k}{2-p}}t^{\frac{1}{2-p}}
        \right)
        <\infty.
    \end{equation*}
    Hence $M_q(\sigma)$ has bounded value on $\sigma\in [1,2]$.
    Applying Lemma \ref{lemma:Mariano_lemma} for $\lambda=1$ and $\lambda^\prime=2$, it follows that
    \begin{multline}\label{eq:thm_result_1}
        \sup_{0<\tau\leq t}\int_{r \leq |x|\leq R}|x|^q u dx
        \leq
        C\int_{\frac{1}{2}r \leq |x|\leq 2R}|x|^q u_0 dx\\
        +
        C t^{\frac{1}{2-p}}\left(
            r^{-\frac{p(N)}{(p-1)(2-p)}} + R^{-\frac{p(N)}{(p-1)(2-p)}},
        \right)
    \end{multline}
    for some $C:=C(N,p)$.
    Thus we prove inequality \eqref{eq:thm_result_ineq} for the case where the initial data $u_0$ is in $C^{\infty}_c(\mathbb{R}^N)$.

    Let us finally address the inequality \eqref{eq:thm_result_ineq} in the general case where $\mu\in \mathcal{M}^+$ is taken as initial data.
    Let $m\in C^{\infty}_c(\mathbb{R}^N)$ be a mollifier function with $\mathrm{supp}(m)\subset B_1, m\geq 0$ and $\int_{\mathbb{R}^N}m dx=1$ and $m_n(x)= n^N m(nx)$.
    Let $\xi_n\in C^{\infty}_c(\mathbb{R}^N)$ be a smooth cut-off function such that $\xi_n=1$ in $B_n$ and $\xi=0$ in $B^c_{2n}$.
    Let $u_{0n}\in C^{\infty}_c(\mathbb{R}^N)$ be defined as follows:
    \begin{equation*}
        u_{0n}(x)
        =
        \xi_n(x)\left(m_n * \mu\right)(x)
        =
        \xi_n(x)\int_{\mathbb{R}^N} m_n(x-y)d\mu(y).
    \end{equation*}
    Obviously, $u_{0n}\geq 0$. 
    Moreover, for any $\rho>0$, since $\mathrm{supp}(m_n(x-\cdot))\subset B_{1/n}(x)$, it follows that  
    \begin{equation}\label{eq:thm_1_initial_data_cond_1}
        \int_{B_{\rho}}u_{0n}dx
        \leq
        \int_{y\in B_{\rho+1/n}}\int_{B_{1/n}(y)\cap B_{\rho}}m_n(x-y)dx d\mu(y)
        \leq
        \mu(B_{\rho+1}).
    \end{equation}
    Furthermore, for any $\phi\in C^{\infty}_c(\mathbb{R}^N)$, take $n$ large enough such that $\mathrm{supp}(\phi)_1\subset B_{n}$, where $\mathrm{supp}(\phi)_1:=\{x\colon d(x,\mathrm{supp}(\phi))\leq 1 \}$.
    Then it follows that
    \begin{equation*}
        \int_{\mathbb{R}^N}\phi u_{0n}dx
        =
        \int_{\mathrm{supp}(\phi)_1}\int_{B_{1/n}(y)}\xi_n(x)\phi(x)m_n(x-y)dx d\mu(y)
        =
        \int_{\mathrm{supp}(\phi)_1}(\phi * \tilde{m}_n)(y) d\mu(y),
    \end{equation*}
    where $\tilde{m}_n(x)=m_n(-x)$.
    By the uniform convergence of $\phi * \tilde{m}_n$ to $\phi$ on compact sets, see \citep[Proposition 4.21]{brezis_functional}, it follows that
    \begin{equation}\label{eq:thm_1_initial_data_cond_2}
        \lim_{n\rightarrow \infty}\int_{\mathbb{R}^N}\phi u_{0n}dx
        =
        \int_{\mathbb{R}^N}\phi d\mu.
    \end{equation}
    Let $u_n$ be the unique positive weak solution of \eqref{eq:p_laplace} and \eqref{eq:initial_data_1} with initial data $u_{0n}$ for $n\geq 1$.
    Then by inequality \eqref{eq:Junning_2} of Lemma \ref{lemma:Junning_Lemma} and \eqref{eq:thm_1_initial_data_cond_1}, it follows that $(u_n)_{n}$ is locally equibounded in $S_T$, that is, for any $\varepsilon>0$ and bounded domain $\Omega\subset \mathbb{R}^N$, $(u_n)_n$ is equibounded in $\Omega_{\varepsilon}:=\Omega \times [\varepsilon,T]$.  
    Then by \citep[Theorem 1]{chen1988}, it follows that $(u_n)_n$ is uniformly equicontinuous in $\Omega_{\varepsilon}$ for all $\varepsilon>0$ and all bounded domains $\Omega\subset \mathbb{R}^N$.
    By diagonalization procedure, we can find a subsequence of $(u_n)$, which is relabelled by $n$, such that
    \begin{equation}\label{eq:thm_approximation}
        u_n,\nabla u_n \rightarrow u, \nabla u \quad \text{uniformly on every compact subset of } S_T,
    \end{equation}
    and $u$ is a positive weak solution of \eqref{eq:p_laplace} and \eqref{eq:initial_data_2} with initial data $\mu$, see \citep[Theorem III.8.1]{dibenedetto1990}.

    Now for any $R>r>0$ and $0<t\leq T$, by inequality \eqref{eq:thm_result_1}, it follows that for any $\tau\in (0,t]$,
    \begin{multline}\label{eq:thm_proof_1}
        \int_{r\leq |x|\leq R}|x|^q u_n(x,\tau)dx
        \leq
        C \int_{\frac{1}{2}r\leq |x|\leq 2 R}|x|^q u_{0n}dx\\
        +
        Ct^{\frac{1}{2-p}}\left(
            r^{-\frac{p(N)}{(p-1)(2-p)}} + R^{-\frac{p(N)}{(p-1)(2-p)}}
        \right).
    \end{multline}
    For the left hand side of \eqref{eq:thm_proof_1}, by dominated convergence and \eqref{eq:thm_approximation}, it follows that
    \begin{equation}\label{eq:thm_proof_2}
        \lim_{n\rightarrow \infty}\int_{r\leq |x|\leq R}|x|^q u_n(x,\tau)dx
        =
        \int_{r\leq |x|\leq R}|x|^q u(x,\tau)dx.
    \end{equation}
    For the first term on the right hand side of \eqref{eq:thm_proof_1}, let $\varepsilon>0$ and $\phi\in C_c(\mathbb{R}^N)$ such that $\phi=|x|^q$ on $\{\frac{1}{2}r\leq |x|\leq 2R\}$ and $\phi=0$ on $\{\frac{1}{2}r-\varepsilon\leq |x|\leq 2R+\varepsilon\}^c$.
    Then taking $n\rightarrow \infty$ and by \eqref{eq:thm_1_initial_data_cond_2}, it follows that
    \begin{equation}\label{eq:thm_proof_3}
        \limsup_{n\rightarrow \infty}\int_{\frac{1}{2}r\leq |x|\leq 2R}|x|^q u_{0n}dx
        \leq
        \int_{\mathbb{R}^N}\phi d\mu
        \leq
        \int_{\frac{1}{2}r-\varepsilon\leq |x|\leq 2R+\varepsilon}|x|^q d\mu.
    \end{equation}
    For $\varepsilon\searrow 0$ together with \eqref{eq:thm_proof_2} and \eqref{eq:thm_proof_3}, it follows that for any $\tau\in (0,t]$,
    \begin{multline*}
        \int_{r\leq |x|\leq R}|x|^q u(x,\tau)dx
        \leq
        C \int_{\frac{1}{2}r\leq |x|\leq 2 R}|x|^q d\mu\\
        +
        Ct^{\frac{1}{2-p}}\left(
            r^{-\frac{p(N)}{(p-1)(2-p)}} + R^{-\frac{p(N)}{(p-1)(2-p)}}
        \right).
    \end{multline*}
    Taking the supremum over $\tau\in (0,t]$ on the left hand side yields the desired result.

    We are left to show uniform boundedness of $q$-moment and inequality \eqref{eq:q_moment_uniformly_integrable}.  
    Assume that $\mu$ is in $\mathcal{M}^+$ with finite total mass and finite $q$-moment, and $p\in(p_c,2)$ is such that $p(N)>0$.
    By inequality \eqref{eq:Junning_1} in Lemma \ref{lemma:Junning_Lemma} and similar arguments, see also \citep[Theorem 1]{zhao1995}, it follows that for any $t\in (0,T]$ and $r>0$,
    \begin{equation}\label{eq:thm_proof_4}
        \sup_{0<\tau\leq t}\int_{B_r}u(x,\tau)dx
        \leq
        C\int_{B_{2r}}d\mu + Cr^{-\frac{k}{2-p}}t^{\frac{1}{2-p}}.
    \end{equation}
    Taking $r=1$ and together with \eqref{eq:thm_result_ineq}, it follows that for any $R>1$
    \begin{multline*}
        \sup_{0<\tau\leq t}\int_{B_R}|x|^q u(x,\tau)dx
        \leq
        \sup_{0<\tau\leq t}\int_{|x|<1}|x|^q u(x,\tau)dx
        +
        \sup_{0<\tau\leq t}\int_{1\leq |x|\leq R}|x|^q u(x,\tau)dx\\
        \leq
        C\int_{\frac{1}{2}\leq |x|\leq 2R}|x|^qd\mu
        +
        C\int_{|x|\leq 2}d\mu
        +
        Ct^{\frac{1}{2-p}}\left(
            1 + R^{-\frac{p(N)}{(p-1)(2-p)}}
        \right).
    \end{multline*}
    Since $p(N)>0$, taking $R\rightarrow \infty$ and it follows that
    \begin{equation*}
        \sup_{0<\tau\leq t}\int_{\mathbb{R}^N}|x|^q u(x,\tau)dx
        \leq
        C \int_{|x|\geq \frac{1}{2}}|x|^qd\mu + C\int_{|x|\leq 2}d\mu
        +
        Ct^{\frac{1}{2-p}}.
    \end{equation*}
    From this we obtain that $\{\mu_{\tau}=u(x,\tau)dx\colon \tau\in [0,t]\}$ has uniformly bounded $q$-moment for any $t>0$. 
    For inequality \eqref{eq:q_moment_uniformly_integrable}, taking $R\rightarrow 0$ for both side of \eqref{eq:thm_result_ineq} and it follows that
    \begin{equation*}
        \sup_{0<\tau\leq t}\int_{|x|\geq r}|x|^q u(x,\tau)dx
        \leq
        C \int_{|x|\geq \frac{1}{2}r}|x|^q d\mu
        +
        C t^{\frac{1}{2-p}}r^{-\frac{p(N)}{(p-1)(2-p)}}.
    \end{equation*}
    Since $\mu$ has finite $q$-moment, taking $r\rightarrow \infty$ and it follows that
    \begin{equation*}
        \lim_{r\rightarrow \infty}\sup_{0<\tau\leq t}\int_{|x|\geq r}|x|^q u(x,\tau)dx
        =0.
    \end{equation*}
\end{proof}

\section{Mass conservation and weak convergence}
In \citep{fino2014}, the authors show that mass conservation of any positive weak solution of \eqref{eq:p_laplace} and \eqref{eq:initial_data_2} with initial data $\mu$ in $\mathcal{M}^{+}_f$ with compact support. 
In this section, we will show that the weak solution constructed in Theorem \ref{theorem:moment_estimate} with initial data $\mu$ in $\mathcal{M}^+$ with finite total mass and finite $q$-moment, preserves mass.
As a by-product, measure $\mu_t:=u(x,t)dx$ converges to $\mu$ weakly.

\begin{theorem}[Mass Conservation]\label{thm:mass_conserv}
    Let $p\in (p_c ,2)$ and $\mu$ be a positive finite Radon measure in $\mathbb{R}^N$ and $u$ is the positive weak solution in Theorem \ref{theorem:moment_estimate}.
    If $\mu$ has $q$-finite moment, then $u$ preserves mass, that is,
    \begin{equation}
        \int_{\mathbb{R}^N}u(x,t)dx
        =
        \int_{\mathbb{R}^N}d\mu, \quad \text{for all}\quad 0<t\leq T.
    \end{equation}
\end{theorem}

\begin{corollary}[Weak Convergence]\label{corollary:weak_convergence}
    Let $\mu$ be a positive finite Radon measure in $\mathbb{R}^N$ with finite $q$-moment and $u$ be the positive weak solution in Theorem \ref{theorem:moment_estimate}.
    Then the measure $\mu_t:=u(x,t)dx$ converges weakly to the initial data $\mu$ as $t$ goes to $0$.
\end{corollary}

We start with a lemma providing gradient estimate of the weak solution of \eqref{eq:p_laplace} with respect to $q$-moment of smooth and compact-supported initial data on the annulus.
\begin{lemma}\label{lemma:ge_annulus}
    Let $u$ be the positive weak solution of \eqref{eq:p_laplace} with initial data $u_0$ in $C^{\infty}_c(\mathbb{R}^N)$ and $u_0\geq 0$.
    Then there exists a positive constant $C:=C(N,p)$ such that for any $0<t\leq T$ and $R>r>0$, it holds that
    \begin{multline}
        \int_{0}^{t}\int_{r\leq |x|\leq R}|\nabla u|^{p-1}dxd\tau
        \\
        \leq
        C t^{\frac{1}{p}}r^{-\frac{1}{p-1}}\left(
            r^{-\frac{p(N)}{p(p-1)}}
            +
            R^{-\frac{p(N)}{p(p-1)}}
        \right)
        \left( M_q\left(u_0,A_{r/4}^{4R}\right)\right)^{\frac{2(p-1)}{p}}\\
        +
        C t^{\frac{1}{2-p}} r^{-\frac{1}{p-1}}\left(
            r^{-\frac{2p(N)}{p(2-p)}} + R^{-\frac{2p(N)}{p(2-p)}}
        \right)
        \left(
            r^{-\frac{p(N)}{p(p-1)}} + R^{-\frac{p(N)}{p(p-1)}}
        \right),
    \end{multline}
    where $M_q(u_0, A):=\int_{A}|x|^q u_0 dx$.
\end{lemma}
\begin{proof}
    Let $R>r>0$ and denote by $A_{\sigma}$ the annulus $A_{r/\sigma}^{R \sigma}$ for $\sigma\geq 1$.
    Let $\xi$ be a smooth cut-off functions such that $\xi=1$ on $A_1$ and $\xi=0$ on $A^c_2$, and $|\nabla \xi|\leq 2r^{-1}$ on $A_{r/2}^{r}$ and $|\nabla \xi|\leq R^{-1}$ on $A_R^{2R}$.
    Since $|x|\geq r$ on $A_1$, it follows that
    \begin{equation}\label{eq:ge_annulus_1}
        \int_{0}^{t}\int_{A_1}|\nabla u|^{p-1}dxd\tau
        \leq
        r^{-\frac{1}{p-1}}\int_{0}^{t}\int_{A_2} |\nabla u|^{p-1}\xi^{p-1}|x|^{\frac{1}{p-1}}dxd\tau.
    \end{equation}
    We consider the right hand side of \eqref{eq:ge_annulus_1} for different cases of $p$:
    \begin{enumerate}[label=\textit{Case \arabic*:}, fullwidth]
        \item If $p>1/(p-1)$, then by inequality \eqref{eq:thm_1_8} in which we choose $\sigma=1$ and $\sigma^{\prime}=2$, it follows that
            \begin{multline}\label{eq:ge_annulus_2}
                \int_{0}^{t}\int_{A_2}|\nabla u|^{p-1}\xi^{p-1}|x|^{\frac{1}{p-1}}dxd\tau
                \leq
                C t \left( r^{1-k-p+\frac{1}{p-1}} + R^{1-k-p+\frac{1}{p-1}}\right)M_q(2)^{p-1}\\
                + 
                C t^{\frac{1}{p}}\left(\max\left\{r^{-\frac{p(N)}{(p-1)(2-p)}}, R^{-\frac{p(N)}{(p-1)(2-p)}}\right\}\right)^{\frac{2-p}{p}} M_q(2)^{\frac{2(p-1)}{p}},
            \end{multline}
            where $M_q(2)=\sup_{0<\tau\leq t}\int_{A_2}|x|^q u dx$.
            Since that $(p-1)\leq \frac{2(p-1)}{p}$, applying Young's inequality $ab\leq a^{\frac{2}{2-p}} + b^{\frac{2}{p}}$ to the first term on the right hand of \eqref{eq:ge_annulus_2}, yields
            \begin{align*}
                t r^{1-k-p+\frac{1}{p-1}} M_q(2)^{p-1} 
                & =
                \left( t^{\frac{1}{2}}r^{1-k-p+\frac{1}{p-1} + \frac{p(N)}{2(p-1)}}\right)
                \left( t^{\frac{1}{2}}r^{-\frac{p(N)}{2(p-1)}}M_q(2)^{p-1}\right)
                \\
                &\leq
                t^{\frac{1}{2-p}}r^{-\frac{p(N)}{(p-1)(2-p)}} 
                + t^{\frac{1}{p}}r^{-\frac{p(N)}{p(p-1)}}M_q(2)^{\frac{2(p-1)}{p}},
            \end{align*}
            which plugged into \eqref{eq:ge_annulus_2} yields
            \begin{multline}\label{eq:ge_annulus_3}
                \int_{0}^{t}\int_{A_2}|\nabla u|^{p-1}\xi^{p-1}|x|^{\frac{1}{p-1}}dxd\tau
                \leq
                C t^{\frac{1}{2-p}}\left(r^{-\frac{p(N)}{(p-1)(2-p)}} + R^{-\frac{p(N)}{(p-1)(2-p)}}\right)\\
                +
                C t^{\frac{1}{p}}\left(r^{-\frac{p(N)}{p(p-1)}} + R^{-\frac{p(N)}{p(p-1)}}\right) M_q(2)^{\frac{2(p-1)}{p}}.
            \end{multline}

        \item If $p\leq 1/(p-1)$, then by inequality \eqref{eq:thm_3_9} in which we choose $\sigma=1$ and $\sigma^{\prime}=2$, it follows that
            \begin{multline}\label{eq:ge_annulus_4}
                \int_{0}^{t}\int_{A_2}|\nabla u|^{p-1}\xi^{p-1}|x|^{\frac{1}{p-1}}dxd\tau
                \\
                \leq
                C t R^{-\frac{p(N)}{p-1}}M_q(2)^{p-1}
                +
                C t^{\frac{1}{p}}R^{-\frac{p(N)}{p(p-1)}}M_q(2)^{\frac{2(p-1)}{p}}.
            \end{multline} 
            Applying the same Young's inequality as in the previous step to the first term on the right hand side, yields
            \begin{align*}
                tR^{-\frac{p(N)}{p-1}}M_q(2)^{p-1}
                &=
                \left(t^{\frac{1}{2}}R^{-\frac{p(N)}{2(p-1)}}\right)
                \left(t^{\frac{1}{2}} R^{-\frac{p(N)}{2(p-1)}}M_q(2)^{p-1}\right)
                \\
                &\leq
                t^{\frac{1}{2-p}}R^{-\frac{p(N)}{(p-1)(2-p)}}
                +
                t^{\frac{1}{p}} R^{-\frac{p(N)}{p(p-1)}}M_q(2)^{\frac{2(p-1)}{p}},
            \end{align*}
            which plugged into \eqref{eq:ge_annulus_4}, implies
            \begin{multline}\label{eq:ge_annulus_5}
                \int_{0}^{t}\int_{A_2}|\nabla u|^{p-1}\xi^{p-1}|x|^{\frac{1}{p-1}}dxd\tau
                \leq
                Ct^{\frac{1}{2-p}}R^{-\frac{p(N)}{(p-1)(2-p)}}
                +
                C t^{\frac{1}{p}}R^{-\frac{p(N)}{p(p-1)}}M_q(2)^{\frac{2(p-1)}{p}}.
            \end{multline}
    \end{enumerate}
    Hence, from \eqref{eq:ge_annulus_3} and \eqref{eq:ge_annulus_5}, we obtain that for any $p\in (p_c,2)$, it holds that
    \begin{multline}
        \int_{0}^{t}\int_{A_2}|\nabla u|^{p-1}\xi^{p-1}|x|^{\frac{1}{p-1}}dxd\tau
        \leq
        C t^{\frac{1}{2-p}}\left(r^{-\frac{p(N)}{(p-1)(2-p)}} + R^{-\frac{p(N)}{(p-1)(2-p)}}\right)\\
        +
        C t^{\frac{1}{p}}\left(r^{-\frac{p(N)}{p(p-1)}} + R^{-\frac{p(N)}{p(p-1)}}\right) M_q(2)^{\frac{2(p-1)}{p}}.
    \end{multline}
    Let $\alpha_1=-p(N)/[(p-1)(2-p)]$ and $\alpha_2=-p(N)/[p(p-1)]$ and $\alpha_3=-2p(N)/[p(2-p)]$ and denote by $F_i$ the terms $r^{\alpha_i} + R^{\alpha_i}$ for $i=1,2,3$.
    By applying inequality \eqref{eq:thm_result_ineq} in Theorem \ref{theorem:moment_estimate} to $M_q(2)$ and inequality $(a+b)^{\gamma}\leq a^{\gamma} + b^{\gamma}$ for $\gamma=2(p-1)/p$, it follows that
    \begin{multline}\label{eq:ge_annulus_6}
        \int_{0}^{t}\int_{A_2}|\nabla u|^{p-1}\xi^{p-1}|x|^{\frac{1}{p-1}}dxd\tau
        \leq
        C t^{\frac{1}{2-p}}F_1 + C t^{\frac{1}{p}} F_2 \left( \int_{A_4}|x|^q u_0 dx + t^{\frac{1}{2-p}}F_1 \right)^{\frac{2(p-1)}{p}}\\
        \leq
        C t^{\frac{1}{2-p}}\left( F_1 + F_2 \left(F_1\right)^{\frac{2(p-1)}{p}}\right)
        + C t^{\frac{1}{p}}F_2 \left(\int_{A_4}|x|^q u_0dx\right)^{\frac{2(p-1)}{p}}.
    \end{multline}
    By using inequality $(a+b)^{\gamma}\leq a^{\gamma}+b^{\gamma}$ for $\gamma=(2-p)/p$ and $\gamma=2(p-1)/p$ respectively, we have $(F_1)^{(2-p)/p} \leq F_2$ and $(F_1)^{2(p-1)/p}\leq F_3$.
    Then for the first two terms on the right hand side in \eqref{eq:ge_annulus_6}, it follows that
    \begin{equation*}
        F_1 + F_2 \left(F_1\right)^{\frac{2(p-1)}{p}}
        =
        \left(F_1\right)^{\frac{2-p}{p} + \frac{2(p-1)}{p}} + F_2 \left(F_1\right)^{\frac{2(p-1)}{p}}
        \leq
        2 F_2 F_3.
    \end{equation*}
    Plugging it into \eqref{eq:ge_annulus_6}, we obtain
    \begin{multline*}
        \int_{0}^{t}\int_{A_2}|\nabla u|^{p-1}\xi^{p-1}|x|^{\frac{1}{p-1}}dxd\tau
        \leq
        C t^{\frac{1}{2-p}}F_2F_3 + C t^{\frac{1}{p}}F_2 \left(\int_{A_4}|x|^q u_0 dx\right)^{\frac{2(p-1)}{p}}.
    \end{multline*}
    which together with \eqref{eq:ge_annulus_1} yields the result.
\end{proof}

\begin{proof}[Proof of Theorem \ref{thm:mass_conserv}]
    Let $\mu\in \mathcal{M}^+$ with finite total mass be fixed and $(u_{0n})$ be the sequence of $C^{\infty}_c$ functions constructed in the proof of Theorem \ref{theorem:moment_estimate}, and $u_n$ be the positive weak solution of \eqref{eq:p_laplace} with $u_{0n}$ as initial data, and $u$ be the positive weak solution constructed in the proof of Theorem \ref{theorem:moment_estimate}.
    Let $\rho>0$ and $\xi$ be a smooth cut-off function such that $\xi=1$ on $B_{\rho}$ and $\xi=0$ on $B^c_{2\rho}$ and $|\nabla \xi|\leq \rho^{-1}$ on $A_{\rho}^{2\rho}$.
    By using $\phi(x,\tau)=\xi(x)\chi_{[0,T]}(\tau)$ as the test function in \eqref{eq:test_condtion_2}, it follows that for any $0<t\leq T$ and $n\geq 1$,
    \begin{multline}\label{eq:mass_conserv_1}
        \left|
        \int_{B_{2\rho}}\xi(x) u_n(x,t) dx
        -
        \int_{B_{2\rho}}\xi(x)u_{0n}(x)dx
        \right|
        \leq
        \int_{0}^{t}\int_{B_{2\rho}}|\nabla u_n|^{p-1} |\nabla \xi| dxd\tau\\
        \leq
        \rho^{-1} \int_{0}^{t}\int_{\rho\leq |x|\leq 2\rho}|\nabla u_n|^{p-1}dxd\tau.
    \end{multline}
    By Lemma \ref{lemma:ge_annulus} with $r=\rho$ and $R=2\rho$, it follows that
    \begin{multline}\label{eq:mass_conserv_2}
        \left|
        \int_{B_{2\rho}}\xi(x) u_n(x,t) dx
        -
        \int_{B_{2\rho}}\xi(x)u_{0n}(x)dx
        \right| \leq
        \\
        C t^{\frac{1}{p}}\rho^{-1-\frac{1}{p-1}-\frac{p(N)}{p(p-1)}}
        \left\{
            \int_{\frac{1}{4}\rho\leq |x|\leq 8\rho}|x|^q u_{0n}dx
        \right\}^{\frac{2(p-1)}{p}}
        +
        C t^{\frac{1}{2-p}}\rho^{-1-\frac{1}{p-1}-\frac{2p(N)}{p(2-p)}-\frac{p(N)}{p(p-1)}}.
    \end{multline}
    By \eqref{eq:thm_approximation} in the proof of Theorem \ref{theorem:moment_estimate}, it holds that $u_n$ converges to $u$, up to a subsequence, uniformly on each compact subset of $S_T$.
    Together with equality \eqref{eq:thm_1_initial_data_cond_2}, taking $n\rightarrow \infty$ on both side of \eqref{eq:mass_conserv_2} and it follows that
    \begin{multline}\label{eq:mass_conserv_3}
        \left|
        \int_{B_{2\rho}}\xi(x) u(x,t) dx
        -
        \int_{B_{2\rho}}\xi(x)d\mu
        \right|
        \leq
        C t^{\frac{1}{p}}\rho^{-1-\frac{1}{p-1}-\frac{p(N)}{p(p-1)}}
        \left\{
            \int_{\frac{1}{4}\rho\leq |x|\leq 8\rho}|x|^q d\mu
        \right\}^{\frac{2(p-1)}{p}}\\
        +
        C t^{\frac{1}{2-p}}\rho^{-1-\frac{1}{p-1}-\frac{2p(N)}{p(2-p)}-\frac{p(N)}{p(p-1)}}.
    \end{multline}
    For the left hand side of \eqref{eq:mass_conserv_3}, by letting $r\rightarrow \infty$ on both side of \eqref{eq:thm_proof_4} in the proof of Theorem \ref{theorem:moment_estimate}, it follows that $\int_{\mathbb{R}^N}u(x,t)dx<\infty$.
    Together with $|\xi(x)u(x,t)|\leq u(x,t)$, by dominated convergence theorem, it follows that
    \begin{equation}
        \lim_{\rho\rightarrow \infty}
        \left|
        \int_{B_{2\rho}}\xi(x)u(x,t)dx - \int_{B_{2\rho}}\xi(x)d\mu
        \right|
        =
        \left|
        \int_{\mathbb{R}^N}u(x,t)dx - \mu(\mathbb{R}^N)
        \right|.
    \end{equation}  
    For the right hand side of \eqref{eq:mass_conserv_3}, it is easy to check that for $p\in (p_c,2)$,
    \begin{align*}
        -1-\frac{1}{p-1}-\frac{p(N)}{p(p-1)}
        =
        -\frac{(N+3)p - 2N}{p}<0,\\
        -1-\frac{1}{p-1}-\frac{2p(N)}{p(2-p)}-\frac{p(N)}{p(p-1)}
        =
        -\frac{(N+1)p - 2N}{2-p}<0.
    \end{align*}
    Since $\mu$ has finite $q$-moment, the term on the right hand side of \eqref{eq:mass_conserv_3} converges to $0$ as $\rho\rightarrow \infty$.
    So we obtain that for any $0<t\leq T$,
    \begin{equation*}
        \int_{\mathbb{R}^N}u(x,t)dx 
        =
        \mu(\mathbb{R}^N).
    \end{equation*}
\end{proof}

\begin{proof}[Proof of Corollary \ref{corollary:weak_convergence}]
    By Theorem \ref{theorem:moment_estimate} and definition of weak solution of \eqref{eq:p_laplace} and \eqref{eq:initial_data_2}, it follows that measure $\mu_t=u(x,t)dx$ converges vaguely  to initial data $\mu$.
    Since $\mu_t(\mathbb{R}^N)=\mu(\mathbb{R}^N)$ for all $0<t< T$, by the classical result(for example, \citep[Theorem 13.16]{Achim_Prob}, it follows that $\mu_t$ converges to $\mu$ weakly as $t$ goes to $0$.
\end{proof}

\section{Convergence rate of Wasserstein distance}
We address now the convergence rate of the constructed weak solution $u(t,\cdot)dx$ to $\mu$ in the $q$-Wasserstein distance.
We recall that the $q$-Wasserstein distance $W_q(\mu,\nu)$ between finite Borel measures $\mu,\nu$ in $\mathbb{R}^N$ with equal mass is defined as
\begin{equation}
    W^q_q(\mu, \nu)
    =
    \min\left\{
        \int_{\mathbb{R}^N \times \mathbb{R}^N} |x-y|^q d\pi(x,y) \colon \pi\in \Pi(\mu,\nu)
    \right\},
\end{equation}
where $\Pi(\mu,\nu)$ is the family of all Borel measures on $\mathbb{R}^N \times \mathbb{R}^N$ having $\mu$ and $\nu$ as their first and second marginal measures respectively.
\begin{theorem}\label{thm:Wasserstein}
    Let $\mu$ be a finite Radon measure on $\mathbb{R}^N$ and $u$ be the weak solution of \eqref{eq:p_laplace} constructed in Theorem \ref{theorem:moment_estimate} with $\mu$ as initial data.
    If $p\in (p_N,2)$ and $\mu$ has finite $q$-moment, then there exist constant $C:=C(N,p,\mu,T)$ such that for all $t\in [0,T]$,
    \begin{equation}\label{eq:Wasserstein_result_1}
        W^q_q(\mu_t, \mu)
        \leq
        C t^{q-1}.
    \end{equation}
\end{theorem}
To address the proof of Theorem \ref{thm:Wasserstein}, we first establish an auxiliary lemma, prove the theorem in the case where the initial data is in $C^{\infty}_c(\mathbb{R}^N$), and then show the general case by an approximation procedure.
The key ingredients of the proof are to use Brenier-Benamou formulation for $q$-Wasserstein distance and write the parabolic $p$-Laplace equation as a law of mass conservation $\partial_t u = \mathrm{div}(\mathbf{v} u)$ where $\mathbf{v}=|\nabla u|^{p-2}(-\nabla u)/u$.

\begin{lemma}\label{lemma:Wasserstein}
    Let $p\in (p_N,2)$ and $u$ be the positive weak solution of Cauchy problem \eqref{eq:p_laplace} with $u_0\in C_{c}^{\infty}(\mathbb{R}^N)$ and $u_0\geq 0$ as initial data.
    Then for any $t>0$, it holds that
    \begin{multline}
        \int_{0}^{t}\int_{\mathbb{R}^N}|\nabla u|^p u^{-\frac{1}{p-1}} dx d\tau
        \leq
        C\left( \int_{|x|\leq 2}u_0 dx\right)^{2-\frac{1}{p-1}}
        +
        C\left( \int_{|x|\geq 1/2}|x|^q u_0 dx\right)^{2-\frac{1}{p-1}}\\
        +
        C t^{\frac{2p-3}{(p-1)(2-p)}}
    \end{multline}
    for some $C:=C(N,p)$.
\end{lemma}
\begin{proof}
    Let $\varepsilon>0$ and $R>0$.
    Let $\xi$ be a smooth cut-off function such that $\xi=1$ on  $B_{R}$ and $\xi=0$ on $B^c_{2R}$ and $|\nabla \xi|\leq R^{-1}$ on $A_R^{2R}$.
    Let $u_{\varepsilon}:=u+\varepsilon$ and $u_{0\varepsilon}:=u_0+\varepsilon$ and $\psi:=u_{\varepsilon}^{1-1/(p-1)}\xi$.
    Multiplying $\psi$ on \eqref{eq:p_laplace} and integrating it over $B_{2R}\times [0,t]$, by integral by part and similar argument in Lemma \ref{lemma:aux_lemma_1}, it follows that
    \begin{align}\label{eq:ge_v5_1}
        \int_{0}^{t}\int_{B_{2R}}|\nabla u|^{p} u_{\varepsilon}^{-\frac{1}{p-1}}\xi dxd\tau
        &\leq
        C \int_{0}^{t}\int_{B_{2R}}u_{\varepsilon}^{p-\frac{1}{p-1}}|\nabla \xi|^p dxd\tau \notag\\
        &\quad \quad +
        C \left(\int_{B_{2R}}u_{\varepsilon}^{2-\frac{1}{p-1}} \xi dx
            -
        \int_{B_{2R}}u_{0\varepsilon}^{2-\frac{1}{p-1}} \xi dx \right) \notag\\
        &\leq
        C \int_{0}^{t}\int_{B_{2R}}u_{\varepsilon}^{p-\frac{1}{p-1}}|\nabla \xi|^p dxd\tau
        +
        C \int_{B_{2R}}u_{\varepsilon}^{2-\frac{1}{p-1}} dx
    \end{align}
    Note that $p\in (p_N,2)$ implies that $p-1/(p-1)>1$.
    Hence, for the first term on the right hand side of \eqref{eq:ge_v5_1}, by using the similar argument as in the proof of Theorem \ref{theorem:moment_estimate}, it follows that
    \begin{align}\label{eq:ge_v5_2}
        \int_{0}^{t}\int_{B_{2R}}u_{\varepsilon}^{p-\frac{1}{p-1}}|\nabla \xi|^p dx d\tau
        &\leq
        R^{-p}\int_{0}^{t}\int_{A_{R}^{2R}} |x|^{-q\left(p-\frac{1}{p-1}\right)}|x|^{q\left(p-\frac{1}{p-1}\right)}u_{\varepsilon}^{p-\frac{1}{p-1}}dx d\tau \notag\\
        &\leq
        C t R^{q\left( 1-k-p+\frac{1}{p-1}\right)}
        \left(
            \sup_{0<\tau\leq t}\int_{A_R^{2R}}|x|^q u_{\varepsilon} dx
        \right)^{p-\frac{1}{p-1}}.
    \end{align} 
    As for the second term on the right hand side of \eqref{eq:ge_v5_1}, we take $\phi_1(x)=1_{\{|x|<1\}} + |x|^{q(2-1/(p-1))}1_{A_1^{2R}}$ and $\phi_2(x)=1_{\{|x|<1\}} + |x|^{-q(2-1/(p-1))}1_{A_1^{2R}}$.
    By H\"older's inequality and $(a+b)^{\gamma}\leq a^{\gamma} + b^{\gamma}$ for $\gamma=2-1/(p-1)$, it follows that
    \begin{align}\label{eq:ge_v5_4}
        \int_{B_{2R}}u_{\varepsilon}^{2-\frac{1}{p-1}}dx
        &\leq
        \left(
            \int_{B_{2R}} \phi_2^{\frac{p-1}{2-p}} dx
        \right)^{\frac{1}{p-1}-1}
        \left(
            \int_{B_{2R}} \phi_1^{\frac{p-1}{2p-3}} u_{\varepsilon} dx
        \right)^{2-\frac{1}{p-1}} \notag\\
        &\leq
        C \left( \int_{B_1} u_{\varepsilon}dx\right)^{2-\frac{1}{p-1}}
        +
        C \left( \int_{A_1^{2R}}|x|^q u_{\varepsilon} dx \right)^{2-\frac{1}{p-1}}.
    \end{align}
    Plugging \eqref{eq:ge_v5_2} and \eqref{eq:ge_v5_4} into \eqref{eq:ge_v5_1}, taking $\varepsilon \rightarrow 0$ and using monotone convergence theorem as well as identity $1-k-p+1/(p-1)=-p(N)/(p-1)$ yields
    \begin{align}\label{eq:ge_v5_5}
        \int_{0}^{t}\int_{B_{2R}}|\nabla u|^p u^{-\frac{1}{p-1}}\xi dxd\tau
        &\leq
        C t R^{-\frac{p(N)q}{p-1}}
        \left( M_q(t,A_r^{2R})\right)^{p-\frac{1}{p-1}}
        +
        C \left( M_0(t,B_1)\right)^{2-\frac{1}{p-1}} \notag\\
        &+
        C \left( M_q(t,A_1^{2R}) \right)^{2-\frac{1}{p-1}},
    \end{align}
    where $M_q(t,A):=\sup_{0<\tau\leq t}\int_{A}|x|^q u(x,\tau) dx$.
    Using Inequality \eqref{eq:thm_result_ineq} from Theorem \ref{theorem:moment_estimate} and \eqref{eq:Junning_1} from \ref{lemma:Junning_Lemma} in the right hand side of \eqref{eq:ge_v5_5}, it follows that
    \begin{multline}\label{eq:ge_v5_6}
        \int_{0}^{t}\int_{B_{R}}|\nabla u|^p u^{-\frac{1}{p-1}} dx d\tau
        \leq
        \int_{0}^{t}\int_{B_{2R}}|\nabla u|^p u^{-\frac{1}{p-1}} \xi dx d\tau\\
        \leq
        C \left( M_0(0,B_2)\right)^{2-\frac{1}{p-1}}
        +
        C t R^{-\frac{p(N)q}{p-1}} \left(M_q(0,A_{R/2}^{4R}) \right)^{p-\frac{1}{p-1}}\\
        +
        C \left( M_q(0,A_{1/2}^{4R}) dx\right)^{2-\frac{1}{p-1}}
        +
        C t^{\frac{2p-3}{(p-1)(2-p)}}\left(
            1 + R^{-\frac{p(N)}{(p-1)(2-p)}}
        \right),
    \end{multline}
    where $M_q(0,A):=\int_{A}|x|^q u_0 dx$.
    Since $p(N)>0$, taking $R\rightarrow \infty$ on both side of \eqref{eq:ge_v5_6}, monotone convergence theorem and $q$-finite moment of $u_0$ yields the result.
\end{proof}

\begin{proof}[Proof of Theorem \ref{thm:Wasserstein}]
    We first show that inequality \eqref{eq:Wasserstein_result_1} holds in the case where $u(x,\tau)$ is the positive weak solution of \eqref{eq:p_laplace} with $u_0$ in $C^{\infty}_c(\mathbb{R}^N)$ and $u_0\geq 0$ as initial data.
    Let $\mu=u_0dx$ and $\mu_t=u(x,t)dx$.
    By Benamou-Brenier formulation for $q$-Wasserstein, see \citep[Theorem 5.28]{Santambrogio_OT} and \citep{BB2000}, it follows that
    \begin{equation*}\label{eq:Wasserstein_1}
        W^q_q(\mu, \mu_t)
        =
        \inf\left\{
            \int_{0}^{1}\int_{\mathbb{R}^N} |v_{\tau}|^q \rho_{\tau} dx d\tau \colon \partial_{\tau}\rho_{\tau} + \mathrm{div}(\rho_{\tau} v_{\tau})=0, \rho_0=u_0, \rho_{1}=u(t)
        \right\}.
    \end{equation*}
    Rescaling as $\tilde{\rho}(x,\tau)=\rho(x,t^{-1}\tau)$ and $\tilde{v}(x,\tau)=t^{-1}v(x,t^{-1}\tau)$ and changing variable yields
    \begin{multline*}
        W^q_q(\mu,\mu_t)
        =\\
        \inf\left\{
            t^{q-1}\int_{0}^{t}\int_{\mathbb{R}^N} |\tilde{v}_{\tau}|^q \tilde{\rho}_{\tau} dx d\tau \colon \partial_{\tau}\tilde{\rho}_{\tau} + \mathrm{div}(\tilde{\rho}_{\tau}\tilde{v}_{\tau})=0, \tilde{\rho}_{0}=u_0, \tilde{\rho}_{t}=u(t)
        \right\}.
    \end{multline*}
    Note that by the property of singular $p$-Laplace equation, $u(x,\tau)>0$ for all $\tau\in (0,t]$. 
    So choosing $\tilde{v}:=-|\nabla u|^{p-2}\frac{\nabla u}{u}$ and $\tilde{\rho}=u$ for $\tau\in (0,t]$, it follows that
    \begin{equation}\label{eq:Wasserstein_2}
        W^q_q(\mu,\mu_t)
        \leq
        t^{q-1}\int_{0}^{t}\int_{\mathbb{R}^N}|\nabla u|^p u^{-\frac{1}{p-1}} dx d\tau.
    \end{equation}
    By Lemma \ref{lemma:Wasserstein}, the inequality \eqref{eq:Wasserstein_result_1} follows.

    We are left to show that inequality \eqref{eq:Wasserstein_result_1} holds in the case where $u$ is the positive weak solution of the Cauchy problem \eqref{eq:p_laplace} with finite Radon measure with finite $q$-moment $\mu$ as initial data, constructed in the proof of Theorem \ref{theorem:moment_estimate}.
    Let $(u_{0n})=((m_n*\mu)\xi_n)$ be the sequence of smooth initial data in $C^{\infty}_{c}(\mathbb{R}^N)$ defined in the proof of Theorem \ref{theorem:moment_estimate} and $u_n$ be the corresponding positive weak solution, and $(u_{n_k})_k,(u_{0n_k})$ be the subsequence of $(u_n),(u_{0n})$ such that $u_{nk}$ and $\nabla u_{nk}$ converges to $u$ and $\nabla u$ uniformly on all compact subset of $S_T$.
    Let $\varepsilon>0$ be fixed.
    For $W_q(\mu,\mu_t)$, by triangle inequality, it follows that
    \begin{multline}\label{eq:Wasserstein_3}
        2^{1-q} W^q_q(\mu,\mu_t)
        \leq
        W^q_q(\mu, \mu^{0n_k}) + W^q_q(\mu^{0n_k}, \mu^{n_k}_t )
        + W^q_q(\mu^{n_k}_tdx, \mu_t),
    \end{multline}
    where $d\mu^{0n}=u_{0n}dx$ and $d\mu^{n}_t=u_{n}(x,t)dx$.

    As for the first term on the right hand side of \eqref{eq:Wasserstein_3}, we claim that $\mu^{0n}$ converges to $\mu$ weakly and that $\sup_{n}\int_{|x|\geq R}|x|^p u_{0n}dx \rightarrow 0$ as $R\rightarrow \infty$.
    Indeed, for all $n\in \mathbb{N}$ and $R>0$, by construction of $u_n$ and Fubini's theorem, it follows that
    \begin{align}\label{eq:Wasserstein_4}
        \int_{|x|\geq R}|x|^q u_{0n}dx
        &=
        \int_{|x|\geq R} |x|^q \xi_n(x)\int_{y\in B_{1/n}(x)}m_n(x-y)d\mu(y)dx \notag\\
        &=
        \int_{|y|\geq (R-1)^+}\int_{x\in B_{1/n}(y)}m_n(x-y)\xi_n(x)|x-y+y|^q dxd\mu(y)\notag\\
        &\leq
        2^{q-1}\int_{|y|\geq (R-1)^+} |y|^q d\mu(y) +  2^{q-1} n^{-q}\int_{|y|\geq (R-1)^+}d\mu(y).
    \end{align}
    Hence, $\sup_{n}\int_{|x|\geq R}|x|^q u_{0n}dx\rightarrow 0$ as $R\rightarrow \infty$.
    Choose $R>0$ large enough such that $\sup_n\int_{|x|\geq R}|x|^q u_{0n}dx<\varepsilon$ and $\mu(B^c_R)<\varepsilon$ and let $\xi$ in $C_c(\mathbb{R}^N)$ such that $\xi=1$ on $B_R$ and $\xi=0$ on $B^c_{2R}$.
    Since $\mu_{0n}$ converges to $\mu$ vaguely by \eqref{eq:thm_1_initial_data_cond_2}, choose $n$ large enough such that
    \begin{multline}\label{eq:Wasserstein_5}
        \left|\int_{\mathbb{R}^N}u_{0n}dx - \mu(\mathbb{R}^N) \right|
        \leq
        \left| \int_{\mathbb{R}^N}\xi u_{0n}dx - \int_{\mathbb{R}^N}\xi d\mu \right|
        +
        \int_{|x|\geq R}u_{0n}dx + \mu(B^c_R)
        \leq
        3\varepsilon.
    \end{multline}
    Hence, $\int_{\mathbb{R}^N}u_{0n}dx \rightarrow \mu(\mathbb{R}^N)$ as $n\rightarrow \infty$.
    By classical result \citep[Theorem 13.16]{Achim_Prob}, $u_{0n}dx$ converges to $\mu$ weakly.
    Together with \eqref{eq:Wasserstein_4}, by \citep[Theorem 7.12]{Villani_book_2003}, it follows that $W_q(\mu^{0n}, \mu)\rightarrow 0$ as $n\rightarrow \infty$.

    For the third term on the right hand side of \eqref{eq:Wasserstein_3}, we claim that $\mu^{n_k}_t$ converges to $\mu_t$ weakly and that $\sup_{k}\int_{|x|\geq R}|x|^p u_{n_k}dx\rightarrow 0$ as $R\rightarrow\infty$.
    Indeed, by inequality \eqref{eq:thm_result_ineq} of Theorem \ref{theorem:moment_estimate}, it follows that
    \begin{equation}\label{eq:Wasserstein_6}
        \sup_{k}\int_{|x|\geq R}|x|^q u_{n_k}(x,t)dx
        \leq
        C\sup_{k}\int_{|x|\geq R/2}|x|^q u_{0n_k}dx + C t^{\frac{1}{2-p}}R^{-\frac{p(N)}{(p-1)(2-p)}}.
    \end{equation} 
    Applied to \eqref{eq:Wasserstein_4}, it follows that $\sup_{k}\int_{|x|\geq R}|x|^q u_{n_k}dx \rightarrow 0$ as $R\rightarrow \infty$.
    Furthermore, using inequality \eqref{eq:Wasserstein_6} as well as \eqref{eq:q_moment_uniformly_integrable}, choose $R>0$ large enough such that $\sup_{k}\int_{|x|\geq R}|x|^q u_{n_k}(x,t)dx<\varepsilon$ and $\int_{|x|\geq R}|x|^q u(x,t)dx<\varepsilon$.
    Then since $u_{n_k}$ converges to $u$ uniformly on all compact subsets on $S_T$, choose $k\in \mathbb{N}$ large enough, such that
    \begin{multline*}
        \left| \int_{\mathbb{R}^N}u_{n_k}(x,t)dx - \int_{\mathbb{R}^N}u(x,t)dx \right|
        \leq
        \left| \int_{B_R}u_{n_k}(x,t)dx - \int_{B_R}u(x,t)dx \right|\\
        +
        \int_{|x|\geq R}|x|^q u_{n_k}(x,t)dx + \int_{|x|\geq R}|x|^q u(x,t)dx
        \leq
        3\varepsilon.
    \end{multline*}
    By the same argument as previously, it follows that $W_q(\mu^{n_k}_t, \mu_t)\rightarrow 0$ as $k\rightarrow \infty$.

    As for the second term on the right hand side of \eqref{eq:Wasserstein_3}, by result from the first step for smooth initial data and inequality \eqref{eq:Wasserstein_4}, it follows that for any $t\in [0,T]$ it holds
    \begin{multline}\label{eq:Wasserstein_7}
        W^q_q(\mu^{0n_k}, \mu^{n_k}_t) \\
        \leq
        C(N,p)\left\{
            \left(\int_{B_2}u_{0n_k}dx\right)^{2-\frac{1}{p-1}}
            +
            \left(\int_{|x|\geq 1/2}|x|^q u_{0n_k}dx\right)^{2-\frac{1}{p-1}}
            +
            T^{\frac{2p-3}{(p-1)(2-p)}}
        \right\}t^{q-1}\\
        \leq
        C(N,p)\left\{ \left(\mu(\mathbb{R}^N)\right)^{2-\frac{1}{p-1}} + \left(\int_{\mathbb{R}^N}|x|^q d\mu\right)^{2-\frac{1}{p-1}} + T^{\frac{2p-3}{(p-1)(2-p)}}\right\}t^{q-1}.
    \end{multline}
    Thus, taking $k\rightarrow \infty$ on both side of \eqref{eq:Wasserstein_3} and together with inequality \eqref{eq:Wasserstein_7}, the result follows.
\end{proof}

\bibliographystyle{plainnat}
\bibliography{biblio}

\begin{thebibliography}{16}
\providecommand{\natexlab}[1]{#1}
\providecommand{\url}[1]{\texttt{#1}}
\expandafter\ifx\csname urlstyle\endcsname\relax
  \providecommand{\doi}[1]{doi: #1}\else
  \providecommand{\doi}{doi: \begingroup \urlstyle{rm}\Url}\fi

\bibitem[Ambrosio et~al.(2019)Ambrosio, Stra, and Trevisan]{ambrosio2019pde}
Luigi Ambrosio, Federico Stra, and Dario Trevisan.
\newblock A pde approach to a 2-dimensional matching problem.
\newblock \emph{Probability Theory and Related Fields}, 173\penalty0
  (1):\penalty0 433--477, 2019.

\bibitem[Benamou and Brenier(2000)]{BB2000}
Jean-David Benamou and Yann Brenier.
\newblock A computational fluid mechanics solution to the monge-kantorovich
  mass transfer problem.
\newblock \emph{Numerische Mathematik}, 84\penalty0 (3):\penalty0 375--393,
  2000.

\bibitem[Benedetto(1993)]{Benedetto1993}
Emmanuele~Di Benedetto.
\newblock \emph{Degenerate Parabolic Equations}.
\newblock Universitext. Springer, 1993.

\bibitem[Brezis(2010)]{brezis_functional}
Haim Brezis.
\newblock \emph{Functional Analysis, Sobolev Spaces and Partial Differential
  Equations}.
\newblock Universitext. Springer-Verlag New York, 2010.

\bibitem[Chen and Di~Benedetto(1988)]{chen1988}
Ya-Zhe Chen and Emmanuele Di~Benedetto.
\newblock On the local behavior of solutions of singular parabolic equations.
\newblock \emph{Archive for Rational Mechanics and Analysis}, 103\penalty0
  (4):\penalty0 319--345, 1988.

\bibitem[Di~Benedetto and Herrero(1990)]{dibenedetto1990}
Emmanuele Di~Benedetto and Miguel~A. Herrero.
\newblock {Non-negative solutions of the evolution p-Laplacian equation.
  Initial traces and Cauchy-problem when $1 < p < 2$}.
\newblock \emph{Archive for Rational Mechanics and Analysis}, 111\penalty0
  (3):\penalty0 225--290, 1990.

\bibitem[Evans and Gangbo(1999)]{evans1999differential}
Lawrence~C Evans and Wilfrid Gangbo.
\newblock \emph{Differential equations methods for the Monge-Kantorovich mass
  transfer problem}.
\newblock Number 653. American Mathematical Soc., 1999.

\bibitem[Fino et~al.(2014)Fino, Düzgün, and Vespri]{fino2014}
Ahmad~Z. Fino, Fatma~Gamze Düzgün, and Vincenzo Vespri.
\newblock Conservation of the mass for solutions to a class of singular
  parabolic equations.
\newblock \emph{Kodai Math. J.}, 37\penalty0 (3):\penalty0 519--531, 10 2014.

\bibitem[Giaquinta(1983)]{mariano1983}
Mariano Giaquinta.
\newblock \emph{Multiple Integrals in the Calculus of Variations and Nonlinear
  Elliptic Systems}.
\newblock Annals of Mathematics Studies 105. Princeton University Press, 1983.

\bibitem[Junning(1995)]{zhao1995}
Zhao Junning.
\newblock The cauchy problem for $u_t = \mathrm{div}(|\nabla u|^{p - 2}\nabla
  u)$ when $2n/(n + 1) < p < 2$.
\newblock \emph{Nonlinear Analysis: Theory, Methods \& Applications},
  24\penalty0 (5):\penalty0 615 -- 630, 1995.

\bibitem[Kamin and V{\'a}zquez(1988)]{kamin1988}
Shoshana Kamin and Juan~Luis V{\'a}zquez.
\newblock Fundamental solutions and asymptotic behaviour for the $ p
  $-laplacian equation.
\newblock \emph{Revista matem{\'a}tica iberoamericana}, 4\penalty0
  (2):\penalty0 339--354, 1988.

\bibitem[Kell(2016)]{kell2016q}
Martin Kell.
\newblock q-{H}eat flow and the gradient flow of the renyi entropy in the
  p-wasserstein space.
\newblock \emph{Journal of Functional Analysis}, 271\penalty0 (8):\penalty0
  2045--2089, 2016.

\bibitem[Klenke(2014)]{Achim_Prob}
Achim Klenke.
\newblock \emph{Probability Theory: a Comprehensive Course}.
\newblock Universitext. Springer-Verlag London, 2nd ed. edition, 2014.

\bibitem[Santambrogio(2015)]{Santambrogio_OT}
Filippo Santambrogio.
\newblock \emph{Optimal Transport for Applied Mathematicians: Calculus of
  Variations, PDEs, and Modeling}.
\newblock Progress in nonlinear differential equations and their applications
  67. Birkhäuser, 1st ed. edition, 2015.

\bibitem[V{\'a}zquez(2006)]{Vazquez2006}
Juan~Luis V{\'a}zquez.
\newblock \emph{Smoothing and Decay Estimates for Nonlinear Diffusion
  Equations: Equations of Porous Medium Type}.
\newblock Oxford Lecture Series in Mathematics and Its Applications. Oxford
  University Press, USA, 2006.

\bibitem[Villani(2003)]{Villani_book_2003}
Cedric Villani.
\newblock \emph{Topics in Optimal Transportation}.
\newblock Graduate Studies in Mathematics 58. American Mathematical Society,
  2003.

\end{thebibliography}
\end{document}